\documentclass{article}
\usepackage[toc,page]{appendix}
\usepackage[utf8]{inputenc}
\usepackage[english]{babel}
\DeclareMathAlphabet{\mathscr}{OT1}{pzc}{m}{it} 
\usepackage{amsmath}
\usepackage{dsfont} 
\usepackage{amssymb}
\usepackage[T1]{fontenc}
\usepackage{mathtools}   
\usepackage{amsfonts}
\usepackage{bbm}
\usepackage{stmaryrd}
\usepackage{mathrsfs}  
\usepackage{graphicx}
\usepackage{bigints}
\usepackage{relsize}
\usepackage[colorinlistoftodos]{todonotes}
\usepackage{amssymb,amsmath,amsthm}
\numberwithin{equation}{section}
\usepackage{dsfont}
\usepackage[shortlabels]{enumitem}
\setlist{labelindent=\parindent,leftmargin=*}
\usepackage{authblk}

\newtheorem{theorem}{Theorem}[section]
\newtheorem{notation}[theorem]{Notation}
\newtheorem{lemma}[theorem]{Lemma}
\newtheorem{proposition}[theorem]{Proposition}
\newtheorem{corollary}[theorem]{Corollary}
\newtheorem{definition}[theorem]{Definition}
\newtheorem{hypothesis}[theorem]{Hypothesis}
\newtheorem{remark}[theorem]{Remark}
\newenvironment{prooff}[1]{\begin{trivlist}
\item {\it \bf Proof}\quad} {\qed\end{trivlist}}
\usepackage{amsmath,amssymb}
\makeatletter
\newsavebox\myboxA
\newsavebox\myboxB
\newlength\mylenA
\newcommand*\xoverline[2][0.75]{%
    \sbox{\myboxA}{$\m@th#2$}%
    \setbox\myboxB\null
    \ht\myboxB=\ht\myboxA%
    \dp\myboxB=\dp\myboxA%
    \wd\myboxB=#1\wd\myboxA
    \sbox\myboxB{$\m@th\overline{\copy\myboxB}$}
    \setlength\mylenA{\the\wd\myboxA}
    \addtolength\mylenA{-\the\wd\myboxB}%
    \ifdim\wd\myboxB<\wd\myboxA%
       \rlap{\hskip 0.5\mylenA\usebox\myboxB}{\usebox\myboxA}%
    \else
        \hskip -0.5\mylenA\rlap{\usebox\myboxA}{\hskip 0.5\mylenA\usebox\myboxB}%
    \fi}
\makeatother

\title{Martingale driven BSDEs, PDEs
and other related deterministic problems}

\author{
Adrien BARRASSO 
\thanks{Université d'Évry Val d'Essonne \\
  Laboratoire de Mathématiques et Modélisation,
23 Bd. de France, 
91037 Évry Cedex,  
F-91128 Palaiseau, France.
 E-mail: \sf adrien.barrasso@univ-evry.fr}    
\qquad\quad
Francesco RUSSO\thanks{ENSTA Paris, Institut Polytechnique de Paris, 
Unit\'e de Math\'ematiques 
appliqu\'ees, 828, boulevard des Mar\'echaux, F-91120 Palaiseau,
 France. E-mail: \sf francesco.russo@ensta-paris.fr}}

\date{November 24th 2020}

\begin{document}
\maketitle

{\bf Abstract.}
We focus on a class of BSDEs driven by a c\`adl\`ag martingale
and the corresponding Markovian BSDEs which arise when
 the randomness of the driver appears through a Markov process.
To those BSDEs we associate a deterministic equation 
which, when the Markov process is  a Brownian diffusion,
is  nothing else but a parabolic semi-linear PDE. We prove existence and uniqueness of a  {\it decoupled mild solution} of the deterministic problem, and give a probabilistic representation of this solution through the aforementioned BSDEs.

\bigskip
{\bf MSC 2010} Classification. 
60H30; 60H10; 35S05; 60J35; 60J60; 60J75.

\bigskip
{\bf KEY WORDS AND PHRASES.} Decoupled mild solutions; Martingale problem;
 c\`adl\`ag martingale;  pseudo-PDE;  Markov processes; backward stochastic differential equation.

\section{Introduction}

In the Brownian context, backward stochastic differential equations (BSDEs) 
were introduced by E. Pardoux and S. Peng in  \cite{parpen90}.
A subclass of BSDEs are said to be {\it Markovian}, if the randomness of
the so called driver $f$ 
depends on a Markovian diffusion $X$, and when the terminal condition depends
 on the terminal value $X_T$.
Those are naturally linked to a parabolic PDE, which constitutes
a particular deterministic problem.
In particular, under reasonable conditions,
 which among other ensure well-posedness,
the solutions of BSDEs produce {\it viscosity} type solutions
for the mentioned PDE. 
In this paper we focus on {\it Pseudo-PDEs}, which are the 
corresponding deterministic problems associated to
  Markovian BSDEs driven by a c\`adl\`ag martingale, when
 the underlying forward process is a general Markov process.
In this case, the concept of a viscosity solution (based on comparison theorems)
is not completely appropriate. For this reason we propose an alternative
notion called {\it decoupled mild solution.} This extends
 the usual formulation of a mild solution, expressed in terms of semigroups,
 which is well-known to the experts
 of PDEs. We establish an existence and uniqueness theorem
among Borel functions having a certain growth condition.

Coming back to Brownian BSDEs, let $s$ be an initial time and $x$ an initial value. A Markovian
 BSDE appears as
\begin{equation}\label{BSDEIntroN}
\left\{\begin{array}{rcl}
X^{s,x}_t &=& x+ \int_s^t \mu(r,X^{s,x}_r)dr +\int_s^t \sigma(r,X^{s,x}_r)dB_r,\quad t\in[s,T]\\
Y^{s,x}_t &=& g(X^{s,x}_T) + \int_t^T f\left(r,X^{s,x}_r,Y^{s,x}_r,Z^{s,x}_r\right)dr  -\int_t^T Z^{s,x}_rdB_r,\quad t\in[s,T],
\end{array}\right.
\end{equation}
where $B$ is a Brownian motion. 
In \cite{peng1991probabilistic} and in \cite{pardoux1992backward} 
previous Markovian BSDE was linked 
to the semilinear PDE
\begin{equation}\label{PDEparabolique}
\left\{
\begin{array}{l}
\partial_tu + \frac{1}{2}Tr\left(\sigma\sigma^\intercal\nabla^2_xu\right)+\mu\cdot\nabla^2_xu
 + f(\cdot,\cdot,u,\sigma\nabla_x u)=0\quad \text{ on } [0,T[\times\mathbbm{R}^d \\
u(T,\cdot) = g. 
\end{array}\right.
\end{equation}
The first link between \eqref{BSDEIntroN} and \eqref{PDEparabolique} was
 established in \cite{peng1991probabilistic}, where the authors showed that when
 the PDE admits a $\mathcal{C}^{1,2}$ solution $u$, then the couple 
  $(Y^{s,x},Z^{s,x})=(u(\cdot,X^{s,x}),\nabla u(\cdot,X^{s,x}))$ solves the BSDE. 
Conversely, if $g$ is continuous (resp. $f$  is continuous 
in $(t,x)$ and is Lipschitz in the third and fourth variable),
 \cite{pardoux1992backward} proved an
 important probability representation result of the (unique)
viscosity solution $u$ of the PDE, via 
the solutions of the Markovian BSDE for each $(s,x)$.
Indeed if  $(Y^{s,x},Z^{s,x})$ is the solution of \eqref{BSDEIntroN}, then $u:(s,x)\longmapsto Y^{s,x}_s$ is a continuous viscosity solution of \eqref{PDEparabolique}.
In  \cite{barles1997sde}, it was shown that, whenever the coefficients belong to some Sobolev spaces, then the function $u$ mentioned above is in fact 
a solution, in the sense of distributions,
 of the PDE.
Later, \cite{BSDEmildPardouxBally} justified that, under certain conditions, $u$ is a  
mild solution of the PDE.

An interesting fact is that, even without further
regularity assumptions made on
 the coefficients of the BSDE,  there exists another function $v$ such that
 $(Y^{s,x},Z^{s,x})=(u(\cdot, X^{s,x}),v(\cdot,X^{s,x}))$, see \cite{el1997backward}. 
In \cite{fuhrman2005generalized} $v$ was associated to $u$ by
 use  of the operator  $\sigma \nabla $ suitably extended.  
However, when the viscosity solution $u$ of the PDE 
has no additional regularity, 
 it is a challenging question to specify
the relation of the function $v$ to $u$, or to the PDE \eqref{PDEparabolique}.
This is the  so called {\it identification problem} and it will
be a central theme in our investigation.

 In \cite{barles1997backward} 
the authors  introduced 
 a new kind of  BSDEs driven by a Brownian motion and a    Poisson random measure. In the Markovian setup, the randomness of its coefficients comes from  an underlying forward process $X$ solving an SDE with jumps.
 They associated this new BSDE with  a non-linear Integro-Partial Differential
 Equation (in short IPDE)  and showed that, under some continuity  and monotonicity conditions on 
the coefficients, the function $u:(s,x)\longmapsto Y^{s,x}_s$ constructed with the BSDEs, is again  a viscosity solution of the IPDE.
Remaining in the framework of Poisson random measures, but without any
diffusion term,
 \cite{Confortola} considered BSDEs driven by marked point processes, see also \cite{BandiniBSDE}.

From a different perspective, BSDEs driven by a general martingale and involving an orthogonal term were studied in  \cite{el1997backward}, \cite{Buckdahn93}, and
 \cite{sant}.
In this paper, we consider a reformulation of such 
 BSDEs, whose given data are
 a continuous increasing process $\hat{V}$,
 a square integrable martingale $\hat{M}$, a terminal condition $\xi$ and a driver $\hat{f}$. 
A solution will be a couple $(Y,M)$ satisfying 
\begin{equation}\label{BSDEIntro}
Y=\xi +\int_{\cdot}^T\hat{f}\left(r,\cdot,Y_r,\frac{d\langle M,\hat{M} \rangle}{d\hat V}(r)\right)d\hat V_r - (M_T-M_{\cdot}),
\end{equation}
where $Y$ is c\`adl\`ag adapted and $M$ is a square integrable martingale.
 We show the existence and the uniqueness of a solution for \eqref{BSDEIntro}.

We will then be interested in a Markov process $(\mathbbm{P}^{s,x})_{(s,x)\in[0,T]\times E}$ taking values in some Polish space $E$ and solving  a martingale problem related to an operator $(\mathcal{D}(a),a)$ and a non-decreasing function $V$. By this we mean that, for any $\phi\in\mathcal{D}(a)$, and $(s,x)\in[0,T]\times E$, $M[\phi]^{s,x}:=\phi(\cdot ,X_{\cdot})-\phi(s,x)-\int_s^{\cdot}a(\phi)(r,X_r)dV_r$ is a $\mathbbm{P}^{s,x}$-martingale. We will fix some
 function $\psi:=(\psi_1,\cdots,\psi_d)\in \mathcal{D}(a)^d$ and at Notation  \ref{N55bis} we will introduce some special
BSDEs driven by a martingale which we will call again
Markovian BSDEs. 

Each BSDE will be indexed by a couple $(s,x)\in[0,T]\times E$, will hold under the probability $\mathbbm{P}^{s,x}$
and will have the form 
\begin{equation}\label{FBSDEIntro}
Y^{s,x}=g(X_T) +\int_{\cdot}^Tf\left(r,X_r,Y^{s,x}_r,\frac{d\langle M^{s,x},M[\psi]^{s,x} \rangle}{dV}(r)\right)dV_r - (M^{s,x}_T-M^{s,x}_{\cdot}),
\end{equation}
where $X$ is the canonical process, $g$ is a Borel function with a
 growth condition and $f$ is Borel, with a growth condition 
with respect to the second variable, and it is Lipschitz with respect 
to the third and fourth variables. 
In most of the examples, we will set $\psi$ to be the identity,
 and $M[\psi]^{s,x}$ will be the martingale part of $X$ under $\mathbbm{P}^{s,x}$. 
We will however also include the case when
 $X$ is not a semimartingale, and in particular  $Id\notin \mathcal{D}(a)^d$.

Those Markovian BSDEs will be linked to the Pseudo-PDE
\begin{equation}\label{PDEIntro}
\left\{
\begin{array}{rccc}
 a(u) + f\left(\cdot,\cdot,u,\Gamma^{\psi}(u)\right)&=&0& \text{ on } [0,T]\times E   \\
 u(T,\cdot)&=&g,& 
\end{array}\right.
\end{equation}
where $\Gamma^{\psi}(u):=\left(a(u\psi_i) - u a(\psi_i) - \psi_i a(u)\right)_{i\in[\![ 1;d]\!]}$, see Definition \ref{MarkovPDE}.
A {\it classical} solution of the Pseudo-PDE will simply be  an element of $\mathcal{D}(a)$ 
fulfilling  \eqref{PDEIntro}.
 We call $\Gamma^{\psi}$ the $\psi$-{\it generalized gradient},
 due to the fact that when $E=\mathbbm{R}^d$, $a=\partial_t+\frac{1}{2}\Delta$ and $\psi_i:(t,x)\longmapsto x_i$ for all $i\in[\![1,d]\!]$ then $\Gamma^{\psi}(u)=\nabla u$. In this particular setup, the forward Markov process is of course a Brownian motion and in this case, the space
 $\mathcal{D}(a) = C^{1,2}([0,T] \times {\mathbb R}^d)$. 

We show the existence of a Borel function $u$ in some extended domain
 $\mathcal{D}(\mathfrak{a})$ such that, for every $(s,x)\in[0,T]\times E$, $Y^{s,x}$ is  a $\mathbbm{P}^{s,x}$-modification of $u(\cdot,X_{\cdot})$
At Definition \ref{D417} we will introduce the notion of
 {\it martingale solution}  
 for the Pseudo-PDE \eqref{PDEIntro},             
operators $\mathfrak{a}$ and $\mathfrak{G}^{\psi}$
extending  $a$ and $\Gamma^\psi$.
 We also show that
 $u$
 is the unique {\it decoupled mild solution} of the same equation.
We explain below that concept of solution, which will be introduced
at Definition \ref{mildsoluv}.

A Borel function $u$ will be called decoupled mild solution
 if there exists an $\mathbbm{R}^d$-valued  Borel function $v:=(v_1,\cdots,v_d)$ such that, for every $(s,x)$,
\begin{equation}\label{MildIntro}
\left\{
    \begin{array}{rcl}
    u(s,x)&=&P_{s,T}[g](x)+\int_s^TP_{s,r}\left[f\left(\cdot,\cdot,u,v\right)(r,\cdot)\right](x)dV_r\\
    u\psi_1(s,x) &=&P_{s,T}[g\psi_1(T,\cdot)](x) -\int_s^TP_{s,r}\left[\left(v_1+ua(\psi_1)-\psi_1f\left(\cdot,\cdot,u,v\right)\right)(r,\cdot)\right](x)dV_r\\
    &\cdots&\\
    u\psi_d(s,x) &=&P_{s,T}[g\psi_d(T,\cdot)](x) -\int_s^TP_{s,r}\left[\left(v_d+ua(\psi_d)-\psi_df\left(\cdot,\cdot,u,v\right)\right)(r,\cdot)\right](x)dV_r,
    \end{array}\right.
\end{equation}
where $P$ is the time-dependent transition kernel
 associated to the Markov canonical class and to the operator $a$, see Notation \ref{N513}.
$v$ 
 coincides with $\mathfrak{G}^{\psi}(u)$ 
and the couple $(u,v)$ will be called solution to the 
{\it identification problem}, see Definition \ref{mildsoluv}.
The intuition behind this notion of solution relies on the fact that the
 equation $a(u)=-f(\cdot,\cdot,u,\Gamma^{\psi}(u))$ can be decoupled into
 the system 
\begin{equation}
\left\{
\begin{array}{ccl}
a(u) &=& - f(\cdot,\cdot,u,v)\\
v_i &=&  \Gamma^{\psi_i}(u),\quad i\in[\![1;d]\!],
\end{array}\right.
\end{equation}
which can be rewritten 
\begin{equation}
\left\{
\begin{array}{ccl}
a(u) &=& - f(\cdot,\cdot,u,v)\\
a(u\psi_i) &=& v_i +ua(\psi_i)- \psi_if(\cdot,\cdot,u,v),\quad i\in[\![1;d]\!].
\end{array}\right.
\end{equation}
Martingale solutions were introduced in \cite{paper1preprint} and 
decoupled mild solutions in \cite{paper2}, but in relation to
a specific type of Pseudo-PDE, for which $v$ is one-dimensional and which does
not include the usual parabolic PDE related to classical BSDEs.
 A first approach to classical solutions for
a general deterministic problem, associated with forward BSDEs 
with applications to the so called {\it F\"ollmer-Schweizer decomposition},
was performed by \cite{laachir}.    

The paper is organized as follows.
In Section \ref{S1} we propose an alternative formulation \eqref{BSDEIntro}
for BSDEs driven by c\`adl\`ag martingales discussed in \cite{sant}:  
 in Theorem \ref{uniquenessBSDE}
 (proved in Appendix \ref{A}),
we state existence and uniqueness for such equations. In Section \ref{S2}, we refer to
 a canonical Markov class and its corresponding  martingale problem.
In Definition \ref{domainextended} we define the extended domain $\mathcal{D}(\mathfrak{a})$;
     in Definition \ref{extended} and Notation \ref{ExtendedGradient},
appear the extended operator $\mathfrak{a}$ and $\mathfrak{G}^{\psi}$.
In Section \ref{S3}, we bring in the Pseudo-PDE \eqref{PDEIntro} (see Definition \ref{MarkovPDE}), the associated Markovian BSDEs \eqref{FBSDEIntro}, 
see Notation \ref{N55bis}. We introduce the notion of martingale solution of the Pseudo-PDE in \eqref{D417} and of decoupled mild solution in Definition \ref{mildsoluv}.
Propositions \ref{MartingaleImpliesMild} and \ref{MildImpliesMartingale}
 show the equivalence between martingale solutions and decoupled mild solutions. 
 Proposition \ref{CoroClassic} states 
that any classical solution is a decoupled mild solution and conversely that 
any decoupled mild solution, belonging to $\mathcal{D}(\Gamma^{\psi}),$ 
is a classical solution up to (what we call) a zero potential set.
Let  $(Y^{s,x},M^{s,x})$ denote the unique solution of the associated BSDE
\eqref{FBSDEIntro}, written as  $BSDE^{s,x}(f,g)$.
In  Theorem \ref{Defuv} we show the existence of some $u\in\mathcal{D}(\mathfrak{a})$
 such that for every $(s,x)\in[0,T]\times E$, $Y^{s,x}$ is a $\mathbbm{P}^{s,x}$-modification of $u(\cdot,X_{\cdot})$ on  $[s,T]$.  Theorem \ref{MainTheorem}
 states that the function $(s,x)\longmapsto Y^{s,x}_s$ is the unique decoupled mild solution of \eqref{PDEIntro}.
 Proposition \ref{MildImpliesBSDE} states that, if the couple $(u,v)$
 satisfies \eqref{MildIntro}, then for any $(s,x)$, the couple $\left(u(t,X_t),\quad u(t,X_t)-u(s,x)+\int_s^{t}f(\cdot,\cdot,u,v)(r,X_r)dV_r\right)_{t\in[s,T]}$ has a $\mathbbm{P}^{s,x}$-version which solves $BSDE^{s,x}(f,g)$ on $[s,T]$.
Finally, in Section \ref{S5}, we study some application examples. In Section \ref{S5a} we deal with parabolic semi-linear PDEs and in Section \ref{S5b} with parabolic semi-linear PDEs with distributional drift.

\section{Preliminary notions and basic notations}\label{preliminaries}

\label{SPrelim}

In this short section we introduce some basic notions, notations and vocabulary 
which will be used in this paper. $T\in\mathbbm{R}_+$ will be a fixed horizon.
\begin{itemize}
	\item For any topological spaces $E$ and $F$, $\mathcal{B}(E)$ will denote the Borel $\sigma$-field of $E$. $\mathcal{C}(E,F)$ (resp. 
$ \mathcal{C}_b(E,F)$, $\mathcal{B}(E,F)$, $\mathcal{B}_b(E,F)$)
 will denote linear
 the space of functions from $E$ to $F$ which are   continuous
(resp. bounded continuous, Borel, bounded Borel).
	\item A filtered probability space $\left(\Omega,\mathcal{F},(\mathcal{F}_t)_{t\in[0,T]},\mathbbm{P}\right)$  will be called  a \textbf{stochastic basis} and said to \textbf{fulfill the usual conditions} if the filtration is complete and right-continuous. 
	\item Given a certain stochastic basis, $\mathcal{H}^2$ will denote 	the  space of square integrable martingales, with the convention that
 indistinguishable elements are identified. $\mathcal{H}^2_0$ will denote the
linear subspace constituted of elements vanishing at zero, and $\mathcal{H}^2_{loc}$ will be the space of locally square integrable martingales.
	\item For any $M,N\in \mathcal{H}^2_{loc}$,  $[M,N]$ will denote  
the \textbf{quadratic covariation} and $\langle M,N\rangle$  their (predictable) \textbf{angle bracket}. If $M=N$ we will use the notations $[M]$ and $\langle M\rangle$.
	\item  $\mathcal{P}ro$ will denote
	the  progressive $\sigma$-field on $[0,T]\times\Omega$. 
	\item 	If $V$ is a non-decreasing process, $dV\otimes d\mathbbm{P}$ will denote the positive measure on 
	$(\Omega\times [0,T],\mathcal{F}\otimes\mathcal{B}([0,T]))$ defined for any $F\in\mathcal{F}\otimes\mathcal{B}([0,T])$ by 
	\\
	$dV\otimes d\mathbbm{P}( F) := \mathbbm{E}\left[\int_0^{T}\mathds{1}_F(\omega,r)dV_r(\omega)\right]$.
	\item If $V$ is a non-decreasing predictable process and  $A$ is a predictable process which is absolutely continuous with respect to $V$, then $\frac{dA}{dV}$ will denote its Radon-Nikodym derivative.
We recall that thanks to 
Proposition 3.2 in \cite{paper1preprint}, this process can be chosen
 to be predictable. 
\end{itemize}

\section{An alternative formulation of BSDEs driven by a 
c\`adl\`ag martingale}\label{S1}

We introduce now an alternative formulation for Backward Stochastic Differential Equations
 driven by a general c\`adl\`ag martingale investigated for instance by \cite{sant}.

From now on, and until the end of this section, we are given
a stochastic basis $\left(\Omega,\mathcal{F},(\mathcal{F}_t)_{t\in[0,T]},\mathbbm{P}\right)$ fulfilling the usual conditions.  
We are also given  some bounded continuous non-decreasing adapted process $\hat V$, we will indicate by $\mathcal{L}^2(d\hat V\otimes d\mathbbm{P})$ the set of 
(up to indistinguishability) progressively measurable processes $\phi$ such that $\mathbbm{E}[\int_0^T \phi^2_rd\hat V_r]<\infty$. 
$\mathcal{L}^{2,cadlag}(d\hat V\otimes d\mathbbm{P})$ will denote the subspace
 of  c\`adl\`ag elements of $\mathcal{L}^{2}(d\hat V\otimes d\mathbbm{P})$.  

We will now fix  an $\mathcal{F}_T$-measurable random variable $\xi$ called the \textbf{final condition}, a square integrable \textbf{reference martingale} $\hat{M}:=(\hat{M}^1,\cdots,\hat{M}^d)$ taking values in $\mathbbm{R}^d$ for some $d\in\mathbbm{N}^*$,  and a \textbf{driver} $\hat{f}:\left([0,T]\times\Omega\right)\times\mathbbm{R}\times\mathbbm{R}^d\longrightarrow\mathbbm{R}$, measurable with respect to  $\mathcal{P}ro\otimes \mathcal{B}(\mathbbm{R})\otimes \mathcal{B}(\mathbbm{R}^d)$.
We will assume that $(\xi,\hat{f},\hat{M})$ satisfies the following.

\begin{hypothesis}\label{HypBSDE}
$\,$
\begin{enumerate}
    \item $\xi\in L^2$;
    \item $\hat{f}(\cdot,\cdot,0,0)\in\mathcal{L}^2(d\hat V\otimes d\mathbbm{P})$; 
    \item There exist positive constants $K^Y,K^Z$ such that, $\mathbbm{P}$ a.s. for all $t,y,y',z,z'$, we have
    \begin{equation}
        |\hat{f}(t,\cdot,y,z)-\hat{f}(t,\cdot,y',z')|\leq K^Y|y-y'|+K^Z\|z-z'\|;
    \end{equation}
    \item $\langle \hat{M}\rangle$ is absolutely continuous with respect to $\hat V$ and $\frac{d\langle \hat{M}\rangle}{d\hat V}$ is bounded.
\end{enumerate}
\end{hypothesis}
We remark that, thanks to Kunita-Watanabe's inequality, the last assumption implies that for any $M\in\mathcal{H}^2_{loc}$, $\langle M,\hat{M}\rangle$ will also be absolutely continuous with respect to $\hat V$.

We will now  formulate precisely our BSDE.
\begin{definition}\label{firstdefBSDE}
We say that a couple
$(Y,M)\in \mathcal{L}^{2,cadlag}(d\hat V\otimes d\mathbbm{P})\times \mathcal{H}^2_0$ is a
solution of $BSDE(\xi,\hat{f}, V,\hat{M})$ if it satisfies
\begin{equation}\label{BSDEcadlag}
Y=\xi +\int_{\cdot}^T\hat{f}\left(r,\cdot,Y_r,\frac{d\langle M,\hat{M} \rangle}{d\hat V}(r)\right)d\hat V_r - (M_T-M_{\cdot})
\end{equation}
in the sense of indistinguishability. 
\end{definition}

The proof of the  theorem below is very similar to the one of Theorem 3.21
 in \cite{paper1preprint}. For the convenience of the reader, it is therefore postponed to Appendix \ref{A}.
\begin{theorem}\label{uniquenessBSDE}
If $(\xi,\hat{f},\hat V,\hat{M})$ satisfies Hypothesis \ref{HypBSDE}, then $BSDE(\xi,\hat{f},\hat V,\hat{M})$ has a unique solution.
\end{theorem}

\begin{remark}\label{BSDESmallInt}
 Let $(\xi,\hat{f},\hat V,\hat{M})$ satisfying 
Hypothesis \ref{HypBSDE}.
We can consider a
BSDE  on a restricted interval $[s,T]$ for some $s \in [0,T[$.
  Theorem \ref{uniquenessBSDE}
 extend easily to this case.
In particular there exists a unique couple of processes 
 $(Y^s,M^s)$, indexed by  $[s,T]$
 such that $Y^s$ is adapted, c\`adl\`ag and satisfies 
$\mathbbm{E}[\int_s^T(Y^s_r)^2d\hat V_r]<\infty$, such that $M^s$ is a martingale vanishing in $s$ and such that
$Y^s = \xi +\int_{\cdot}^T \hat{f}\left(r,\cdot,Y^s_r,
\frac{d\langle M^s,\hat{M}\rangle}{d\hat V}(r)\right)d\hat V_r -(M^s_T-M^s_{\cdot})$
in the sense of indistinguishability on $[s,T]$.
 
 Moreover, if $(Y,M)$ 
denotes the solution of $BSDE(\xi,\hat{f},\hat V,\hat{M})$ then $(Y,M_{\cdot}-M_s)$ and $(Y^s,M^s)$ coincide on $[s,T]$. This follows by an uniqueness
argument resulting by Theorem \ref{uniquenessBSDE} on the time interval
$[s,T]$.
\end{remark}

\begin{remark} \label{RemExtensions}
		\begin{enumerate}\
\item
\cite{sant} considers a BSDE driven by a c\`adl\`ag martingale
  which corresponds to the BSDE \eqref{BSDEIntroN}, where the Brownian
   motion $W$
   is replaced with a martingale $M$ with non-necessarily bounded angular
   bracket $\langle M \rangle$, with a remainder orthogonal martingale $N$.
The solution is given by a triplet $(Y,Z,N)$.
The authors make use of  weighted spaces of the type 
${\mathcal H}^2_{T,\beta}$ and ${\mathcal L}^2_{\beta}$. 
For instance ${\mathcal H}^2_{T,\beta}$ is 
  the space of all progressively measurable processes $\phi$ such that
 $\mathbb{E}{\int_0^T\phi^2_se^{\beta \langle M\rangle_s}}d\langle M\rangle_s<+\infty$.
In particular they find a value for $\beta$ such that existence and
 uniqueness holds within the class of triplets $(Y,Z,N)$ such that $Y, Z \in
 {\mathcal H}^2_{T,\beta}$
and $N \in {\mathcal L}^2_\beta$.
		\item 
Existence and uniqueness theorems for Brownian BSDEs can be also 
stated under more general assumptions than Lipschitz conditions.
In \cite{kobylanski2000backward}, the author
has obtained an existence result for possibly quadratic growth
BSDEs,
when the driver $f$ is of the form $f(t,y,z)=f^1(t,z)y+f^2(t,y,z)$ where
 $f^1$ is bounded a.s., and for all $t,y,z$,
 $|f^2(t,y,z)|\leq K(1+c(|y|)|z|^2)$ for some continuous function $c$. On the other hand the terminal condition $\xi$ is supposed to be bounded.
	\end{enumerate}
\medskip
We believe that several arguments developed in
the two previous items can be adapted to our context.
However, in this paper we have chosen not to explore the validity of Theorem 
\ref{uniquenessBSDE} under more general assumptions along the line of items 1.
 and 2.
It will be the object of future investigations.
\end{remark}

\section{Martingale Problem and canonical Markov classes}\label{S2}

We now introduce the Markov process which will be the forward underlying
 of our  BSDE driven by a c\`adl\`ag martingale. 
 That  process will be defined as the solution of  a  martingale problem described below.

For  details concerning the exact mathematical  framework  for 
our  Markov process, we refer to our previous paper  \cite{paperAF}
about canonical Markov classes and additive functionals.

From now on, $E$ is a Polish space and $\left(\Omega,\mathcal{F},(X_t)_{t\in[0,T]},(\mathcal{F}_t)_{t\in[0,T]}\right)$ denotes the canonical space defined in Notation 3.1 of \cite{paperAF}.  We also fix a  canonical Markov class $(\mathbbm{P}^{s,x})_{(s,x)\in[0,T]\times E}$ associated to a transition kernel $P = (P_{s,t})$ measurable in time as defined in Definitions 3.4, 3.5 and 3.7 in \cite{paperAF}. For any $(s,x)\in [0,T]\times E$, $\left(\Omega,\mathcal{F}^{s,x},(\mathcal{F}^{s,x}_t)_{t\in[0,T]},\mathbbm{P}^{s,x}\right)$ will denote the stochastic basis in which $\mathbbm{P}^{s,x}$-null sets are added to $\mathcal{F}$ and $\mathcal{F}_t$ for all $t$, and which fulfills the usual conditions. 
${\mathbbm E}^{s,x}$ will denote the corresponding expectation 
to $\mathbbm{\ P}^{s,x}$.
If $P_{s,t}$ only depends on $t-s$, $P$ is called
time-homogeneous and we 
we will often
use the notation $P_t$ instead of $P_{0,t}$.
\begin{notation} \label{N513}
          In particular, for any $t \in [0,T]$ and $A\in\mathcal{B}(E)$
\begin{equation}\label{Markov1}
\mathbbm{P}^{s,x}(X_t\in A)=P_{s,t}(x,A),
\end{equation}
and for any $s\leq t\leq u$
\begin{equation}\label{Markov2}
\mathbbm{P}^{s,x}(X_u\in A|\mathcal{F}_t)=P_{t,u}(X_t,A)\quad \mathbbm{P}^{s,x}\text{ a.s.}
\end{equation}
        	Let $s,t$ in $[0,T]$ with $s\leq t$, $x\in E$ and $\phi\in \mathcal{B}(E,\mathbbm{R}).$  If $\phi$ is integrable with respect to $P_{s,t}(x,\cdot)$, then $P_{s,t}[\phi](x)$ will denote its integral.
\end{notation}

We recall two important measurability properties,
essentially stated in \cite{paper2}, even though with
 $V(t) \equiv t$.
\begin{remark} \label{R63}\leavevmode
	\begin{itemize}
		\item Let $\phi\in\mathcal{B}(E,\mathbbm{R})$ be such that for any $(s,x,t)$, 
		$\mathbbm{E}^{s,x}[|\phi(X_t)|]<\infty$, then  $(s,x,t)\longmapsto P_{s,t}[\phi](x)$ is Borel, see Proposition A.11 in \cite{paper2}.
		\item Let $\phi\in\mathcal{L}^1_X$,
		 then $(s,x)\longmapsto\int_s^TP_{s,r}[\phi](x)dV_r$ is Borel, see Lemma A.10 in \cite{paper2}.
	\end{itemize}
\end{remark}

\begin{definition}\label{MartingaleProblem}
 Let  
$V:[0,T]\rightarrow\mathbbm{R}_+$
be a non-decreasing continuous function 
 vanishing at 0.
Let us consider a linear operator
$a:\mathcal{D}(a) \subset \mathcal{B}([0,T]\times E,\mathbbm{R})
 \longrightarrow \mathcal{B}([0,T]\times E,\mathbbm{R})$,
where the domain $\mathcal{D}(a)$ is a linear space.
\\

We say that  $(\mathbbm{P}^{s,x})_{(s,x)\in [0,T]\times E}$  solves the {\bf martingale problem associated to}  $(\mathcal{D}(a),a,V)$ if, for  any 
$(s,x)\in[0,T]\times E$, $\mathbbm{P}^{s,x}$ satisfies the following. 
\begin{description}
\item{(a)} $\mathbbm{P}^{s,x}(\forall t\in[0,s], X_t=x)=1$;
\item{(b)} for every $\phi\in\mathcal{D}(a)$, $\phi(\cdot,X_{\cdot}) - \int_s^{\cdot} a(\phi)(r,X_r)dV_r$, 
$t \in [s,T]$, is  a c\`adl\`ag $(\mathbbm{P}^{s,x},(\mathcal{F}_t)_{t\in[s,T]})$ square integrable martingale.
\end{description}
 The Martingale Problem is said to be \textbf{well-posed} if for any $(s,x)\in[0,T]\times E$, $\mathbbm{P}^{s,x}$ is the unique probability measure satisfying 
those two properties.
\end{definition}
We anticipate that well-posedness for the martingale problem will not
be an hypothesis in the sequel.

\begin{notation}\label{Mphi}
For every $(s,x)\in[0,T]\times E$ and $\phi\in\mathcal{D}(a)$, the process 
\\
$t\mapsto\mathds{1}_{[s,T]}(t)\left(\phi(t,X_{t})-\phi(s,x)-\int_s^{t}a(\phi)(r,X_r)dV_r\right)$ will be denoted $M[\phi]^{s,x}$.
\end{notation}
$M[\phi]^{s,x}$ is a c\`adl\`ag $(\mathbbm{P}^{s,x},(\mathcal{F}^{s,x}_t)_{t\in[0,T]})$ square integrable martingale vanishing on $[0,s]$.

\begin{notation}\label{MphiMAF}
Let $\phi\in\mathcal{D}(a)$. 
 For $0\leq t\leq u\leq T$, we set 
\begin{equation}
M[\phi]^t_u:=
\left\{
\begin{array}{l}
\phi(u,X_u)-\phi(t,X_t)-\int_t^u a(\phi)(r,X_r)dV_r\text{ if }\int_t^u|a(\phi)|(r,X_r)dV_r <\infty, \\
0 \ \text{otherwise}.
\end{array}\right.
\end{equation}
$M[\phi]$ is  a square integrable Martingale Additive Functional (in short MAF), see Definition 4.1 in \cite{paperAF}, whose c\`adl\`ag version  under $\mathbbm{P}^{s,x}$ for every $(s,x)\in[0,T]\times E$, is  $M[\phi]^{s,x}$.
\end{notation}

From now on we fix some $d\in\mathbbm{N}^*$  and a vector $\psi = (\psi_1, \ldots,\psi_d)\in\mathcal{D}(a)^d$.
For any $(s,x)\in[0,T]\times E$, the $\mathbbm{R}^d$-valued martingale $(M[\psi_1]^{s,x},\cdots,M[\psi_d]^{s,x})$ will be denoted $M[\psi]^{s,x}$.

\begin{definition}\label{CarreDuChamp}
For any $\phi_1,\phi_2\in\mathcal{D}(a)$ such that $\phi_1\phi_2\in\mathcal{D}(a)$ we set $\Gamma(\phi_1,\phi_2):=a(\phi_1\phi_2)-\phi_1a(\phi_2)-\phi_2a(\phi_1)$. $\Gamma$ will be called the {\bf carr\'e du champs operator}.
We set 
$\mathcal{D}(\Gamma^{\psi}):=\left\{\phi\in\mathcal{D}(a):
\forall i\in[\![ 1;d]\!],\phi\psi^i\in\mathcal{D}(a)\right\}$ and
we  define the linear operator $\Gamma^{\psi}:\mathcal{D}(\Gamma^{\psi})\longrightarrow\mathcal{B}([0,T]\times E,\mathbbm{R}^d)$ by 
\begin{equation} \label{D45}
   \Gamma^{\psi}(\phi) := \left(\Gamma^{\psi_i}(\phi)\right)_{i\in[\![ 1;d]\!]}:= \left(a(\phi\psi_i) - \phi a(\psi_i) - \psi_i a(\phi)\right)_{i\in[\![ 1;d]\!]}.
\end{equation} 
$\Gamma^{\psi}$ will be called the $\psi$-\textbf{generalized gradient operator}.
\end{definition}
We emphasize that this terminology is justified by the considerations below \eqref{PDEIntro}.
This operator appears in the expression of the angular bracket of the local martingales
 that we have defined.

\begin{proposition}\label{bracketindomain}
If $\phi\in \mathcal{D}(\Gamma^{\psi})$, then for any  $(s,x)\in[0,T]\times E$ and $i\in[\![ 1;d]\!]$ we have
\begin{equation} \label{Ebracket}
\langle M[\phi]^{s,x} , M[\psi_i]^{s,x} \rangle = \int_s^{\cdot\vee s} \Gamma^{\psi_i}(\phi)(r,X_r)dV_r,
\end{equation}
 in the stochastic basis $(\Omega,\mathcal{F}^{s,x},(\mathcal{F}^{s,x}_t)_{t\in[0,T]},\mathbbm{P}^{s,x}).$
\end{proposition}
\begin{proof}
The result follows from a slight modification of the proof of Proposition 4.7 of \cite{paper1preprint} in which $\mathcal{D}(a)$ was assumed to be stable by multiplication and $M[\phi]^{s,x}$ could potentially be a local martingale 
which is not a martingale.
\end{proof}

We will later need the following assumption.

\begin{hypothesis}\label{HypBrackPhi}
For every $i\in[\![ 1;d]\!]$, the Additive Functional  $\langle M[\psi_i]\rangle$ (which is well defined thanks to Corollary 4.9 in \cite{paperAF})  has càdlàg versions which are absolutely continuous with respect to $dV$.
\end{hypothesis}

Taking $\phi = \psi_i$ for some $i\in[\![ 1;d]\!]$ in Proposition \ref{bracketindomain}, yields the following.
\begin{corollary}\label{H2Vloc}	
If $\psi_i^2\in\mathcal{D}(a)$ for all $i\in[\![ 1;d]\!]$, then Hypothesis \ref{HypBrackPhi} is fulfilled.
\end{corollary} 

We will now consider suitable extensions of the domain $\mathcal{D}(a)$.
\\
\\
For any $(s,x)\in[0,T]\times E$ we define the positive bounded \textbf{potential measure} $U(s,x,\cdot)$ on $\left([0,T]\times E,\mathcal{B}([0,T])\otimes \mathcal{B}(E)\right)$ by 
\\
$U(s,x,\cdot):\begin{array}{ccl}
\mathcal{B}([0,T])\otimes \mathcal{B}(E)&\longrightarrow& [0,V_T]\\
A &\longmapsto& \mathbbm{E}^{s,x}\left[\int_s^{T} \mathds{1}_{\{(t,X_t)\in A\}}dV_t\right].
\end{array}$

\begin{definition}\label{zeropotential}
A Borel set $A\subset [0,T]\times E$ will be said to be
 {\bf of zero potential} if, for any $(s,x)\in[0,T]\times E$  we have  $U(s,x,A) = 0$.
\end{definition}
\begin{notation}\label{topo}
Let 
$p > 0 $. We introduce
\\
${\mathcal L}^p_{s,x} :={\mathcal L}^p(U(s,x,\cdot)) =\left\{ f\in \mathcal{B}([0,T]\times E,\mathbbm{R}):\, \mathbbm{E}^{s,x}\left[\int_s^{T} |f|^p(r,X_r)dV_r\right] < \infty\right\}$.
\\
For $p \ge 1$, that classical $\mathcal{L}^p$-space is equipped with the
 seminorm
\\
$\|\cdot\|_{p,s,x}:f\mapsto \left(\mathbbm{E}^{s,x}\left[\int_s^{T}|f(r,X_r)|^pdV_r\right]\right)^{\frac{1}{p}}$.
 We also introduce 
\\
${\mathcal L}^0_{s,x} :={\mathcal L}^0(U(s,x,\cdot)) = \left\{ f\in \mathcal{B}([0,T]\times E,\mathbbm{R}):\, \int_s^{T} |f|(r,X_r)dV_r < \infty\quad\mathbbm{P}^{s,x}\text{ a.s.}\right\}$.
\\
For any $p\ge 0$ we set 
\begin{equation}
\mathcal{L}^p_X =\underset{(s,x)\in[0,T]\times E}{\bigcap}{\mathcal L}^p_{s,x}.
\end{equation}
Let $\mathcal{N}$ be the linear subspace of $\mathcal{B}([0,T]\times E,\mathbbm{R})$ containing all functions which are equal to 0, $U(s,x,\cdot)$ a.e. for every $(s,x)$.
For any $p \ge 0 $,
 we define  the quotient space $L^p_X = \mathcal{L}^p_X /\mathcal{N}$.
If $p \ge 1$, $L^p_X$ can be equipped with the topology generated by the family of semi-norms $\left(\|\cdot\|_{p,s,x}\right)_{(s,x)\in[0,T]\times E}$ which makes it a separate locally convex topological vector space, 
see Theorem 5.76 in \cite{aliprantis}.
\end{notation}

We recall that Proposition 4.13 in \cite{paper1preprint} states the following.
\begin{proposition}\label{uniquenessupto}
Let $f$ and $g$ be in $\mathcal{L}^0_X$.
Then $f$ and $g$ are equal up to a set of zero potential if and only if for any $(s,x)\in[0,T]\times E$, the processes $\int_s^{\cdot}f(r,X_r)dV_r$ and $\int_s^{\cdot}g(r,X_r)dV_r$ are indistinguishable under $\mathbbm{P}^{s,x}$.
Of course in  this case $f$ and $g$ correspond to the same element of $L^0_X$.
\end{proposition}

We introduce now  our notion of \textbf{extended generator}
starting from its domain.
\begin{definition}\label{domainextended}
	We first define the \textbf{extended domain} $\mathcal{D}(\mathfrak{a})$ as the set of functions $\phi\in\mathcal{B}([0,T]\times E,\mathbbm{R})$ for which there exists 
	\\
	$\chi\in\mathcal{L}^0_X$ such that under any $\mathbbm{P}^{s,x}$ the process
	\begin{equation} \label{E45}
	\mathds{1}_{[s,T]}\left(\phi(\cdot,X_{\cdot}) - \phi(s,x) - \int_s^{\cdot}\chi(r,X_r)dV_r \right) 
	\end{equation}
	(which is not necessarily c\`adl\`ag) has a c\`adl\`ag modification in $\mathcal{H}^2_{0}$. 
\end{definition}

A direct consequence of  Proposition 4.15 in \cite{paper1preprint} is the following.
\begin{proposition}\label{uniquenesspsi}
	Let $\phi \in {\mathcal B}([0,T] \times E, {\mathbbm R}).$
	There is at most one (up to zero potential sets)  $\chi 
	\in {\mathcal L}^0_X$ such that
	under any $\mathbbm{P}^{s,x}$, the process defined in \eqref{E45} 
	has a modification which belongs to $\mathcal{H}^2$.
	\\
	If moreover $\phi\in\mathcal{D}(a)$, then $a(\phi)=\chi$ up to zero potential sets. In this case, according to Notation \ref{Mphi},
	for every $(s,x)\in[0,T]\times E$, 
	$M[\phi]^{s,x}$ is the $\mathbbm{P}^{s,x}$ c\`adl\`ag modification
	in $\mathcal{H}^2_{0}$ of
	$\mathds{1}_{[s,T]}\left(\phi(\cdot,X_{\cdot}) - \phi(s,x) - \int_s^{\cdot}\chi(r,X_r)dV_r \right)$.
\end{proposition}

\begin{definition}\label{extended}
	Let $\phi \in \mathcal{D}(\mathfrak{a})$ as in Definition
	\ref{domainextended}.
	We denote again  by $M[\phi]^{s,x}$, the unique
	c\`adl\`ag version
	of the process \eqref{E45} in $\mathcal{H}^2_{0}$.
	Taking Proposition \ref{uniquenessupto} into account, this will not
	generate  any ambiguity with respect
	to Notation \ref{Mphi}.   
	Proposition \ref{uniquenessupto}, also permits to define without ambiguity the operator  
	\begin{equation*}
	\mathfrak{a}:
	\begin{array}{ccl}
	\mathcal{D}(\mathfrak{a})&\longrightarrow& L^0_X\\
	\phi &\longmapsto & \chi.
	\end{array}
	\end{equation*}
	$\mathfrak{a}$ will be called the \textbf{extended generator}.
\end{definition}
\begin{remark}\label{Rextendeda}
$\mathfrak{a}$ extends $a$ in the sense that $\mathcal{D}(a)\subset\mathcal{D}(\mathfrak{a})$ (comparing Definitions \ref{domainextended} and \ref{MartingaleProblem}) and if $\phi\in\mathcal{D}(a)$ then $a(\phi)$ is an element of the class $\mathfrak{a}(\phi)$, see Proposition \ref{uniquenesspsi}.
\end{remark}

We also introduce an extended $\psi$-generalized gradient. 

\begin{proposition} \label{P321}
Assume the validity of Hypothesis \ref{HypBrackPhi}.
	Let $\phi\in\mathcal{D}(\mathfrak{a})$ and $i\in[\![ 1;d]\!]$.  There exists a (unique up to zero-potential sets) function in $\mathcal{B}([0,T]\times E,\mathbbm{R})$ which we will denote $\mathfrak{G}^{\psi_i}(\phi)$ such that under any $\mathbbm{P}^{s,x}$, 
	$\langle M[\phi]^{s,x},M[\psi_i]^{s,x}\rangle=\int_s^{\cdot\vee s}\mathfrak{G}^{\psi_i}(\phi)(r,X_r)dV_r$  up to indistinguishability.
\end{proposition} 
\begin{proof}
We fix $i\in[\![1;d]\!]$. Let $M[\psi_i]$ be the
square integrable MAF (see  4.1 in \cite{paperAF})
presented in Notation \ref{MphiMAF}. 
We introduce the random field $M[\phi] = (M[\phi]_u^t)_{(0\le t \le u\leq T)} $ 
 as follows.
We fix some $\chi$ in the class $ \mathfrak{a}(\phi)$ and set
\begin{equation}
M[\phi]^t_u:=\left\{
\begin{array}{l}
\phi(u,X_u)-\phi(t,X_t)-\int_t^u\chi(r,X_r)dV_r\text{ if }\int_t^u|\chi|(r,X_r)dV_r <\infty,  t\leq u,   \\ 
0 \ \text{elsewhere},
\end{array}\right.
\end{equation}
 We emphasize that, a priori, the function $\chi$ is only in $\mathcal{L}^0_X$ implying that at fixed $t\leq u$, $\int_t^u|\chi|(r,X_r(\omega))dV_r$ is not finite for every $\omega \in \Omega$, but only on a set which is $\mathbbm{P}^{s,x}$-negligible for all $(s,x)\in[0,t]\times E$. 
\\
According to Definition 4.1 in \cite{paperAF} $M[\phi]$ 
is an AF whose c\`adl\`ag version  under $\mathbbm{P}^{s,x}$ is $M[\phi]^{s,x}$.
Of course   $M[\psi_i]^{s,x}$ is the c\`adl\`ag version of $M[\psi_i]$  
 under $\mathbbm{P}^{s,x}$.
\\
By Definition \ref{extended}, since 
$\phi\in\mathcal{D}(\mathfrak{a})$, $M[\phi]^{s,x}$ 
is a square integrable martingale 
for every $(s,x)$, so $M[\phi]$ is a square integrable MAF.
 Then by Corollary \ref{H2Vloc},  the AF $\langle M[\psi_i]\rangle$ is absolutely continuous with respect to $dV$.
The existence of $\mathfrak{G}^{\psi_i}(\phi)$ now
 follows from Proposition 4.14  in \cite{paperAF}.
and the uniqueness follows by Proposition \ref{uniquenessupto}.
\end{proof}

\begin{notation}\label{ExtendedGradient}
If \ref{HypBrackPhi} holds, we can  introduce the linear operator
\begin{equation}
\mathfrak{G}^{\psi}:\begin{array}{ccl}
\mathcal{D}(\mathfrak{a})&\longrightarrow&(L^0_X)^d\\
\phi&\longmapsto& (\mathfrak{G}^{\psi_1}(\phi),\cdots,\mathfrak{G}^{\psi_d}(\phi)),
\end{array}
\end{equation}
which will be called the \textbf{extended $\psi$-generalized gradient}.
\end{notation}

\begin{corollary}\label{RExtendedClassical} 
Let $\phi\in\mathcal{D}(\Gamma^{\psi})$. Then $\Gamma^{\psi}(\phi)=\mathfrak{G}^{\psi}(\phi)$ up to zero potential sets.
\end{corollary}
\begin{proof}
Comparing Propositions \ref{bracketindomain} and \ref{P321}, for every $(s,x)\in[0,T]\times E$ and $i\in[\![ 1;d]\!]$,
  $\int_s^{\cdot\vee s}\Gamma^{\psi_i}(\phi)(r,X_r)dV_r$ and $\int_s^{\cdot\vee s}\mathfrak{G}^{\psi_i}(\phi)(r,X_r)dV_r$ are $\mathbbm{P}^{s,x}$-indistinguishable. We can conclude by Proposition \ref{uniquenessupto}.
\end{proof}
$\mathfrak{G}^{\psi}$ therefore extends $\Gamma^{\psi}$  as well as $\mathfrak{a}$ extends $a$, see Remark \ref{Rextendeda}.

\section{Pseudo-PDEs and associated Markovian type BSDEs driven by a c\`adl\`ag martingale}\label{S3}
 
\subsection{The concepts}

In this section, we keep working in the framework of the previous Section \ref{S2}.

We now introduce a subclass of BSDEs driven by a c\`adl\`ag
 martingale which we will call {\bf Markovian}.
 The process $\hat V$  will be the (deterministic)   function $V$ introduced in
  Definition \ref{MartingaleProblem},  the terminal condition $\xi$ will only depend on the final value of the canonical process $X_T$ and the  randomness of the driver $\hat{f}$  at time $t$ will only depend on $X_t$. In other words, the driver will be of type ${\hat f}(t,\omega, y,z) = f(t,X_t(\omega), y, z)$ where 
 $f:[0,T] \times E \times {\mathbb R} \times {\mathbb R}^d \rightarrow {\mathbb R}$ is a measurable function.

Given $d$ functions $\psi_1,\cdots,\psi_d$ in $\mathcal{D}(a)$, we will set $\hat{M}:=(M[\psi_1]^{s,x},\cdots,M[\psi_d]^{s,x})$.

That BSDE will be connected with the deterministic problem in 
Definition \ref{MarkovPDE}.

We fix an integer $d\in\mathbbm{N}^*$ and some functions $\psi_1,\cdots,\psi_d\in\mathcal{D}(a)$ which in the sequel,  will satisfy the following hypothesis.
\begin{hypothesis}\label{HypPhi}
	For any $i\in[\![ 1;d]\!]$ we have the following.
	\begin{itemize}
		\item Hypothesis \ref{HypBrackPhi} holds;
		\item $a(\psi_i)\in\mathcal{L}^2_X$;
		\item $\mathfrak{G}^{\psi_i}(\psi_i)$ is bounded.
	\end{itemize}
\end{hypothesis}
 \begin{proposition} \label{R57}\leavevmode
   Assume that Hypothesis \ref{HypPhi} holds. Then 
  for every $i\in[\![ 1;d]\!]$, we have the following. 
	\begin{itemize}
		\item For any $(s,x)\in[0,T]\times E$,
 $\hat M := M[\psi]^{s,x}$ satisfies item 4.  of Hypothesis \ref{HypBSDE} with respect to $\hat V:=V$.
		\item for every $(s,x)\in[0,T]\times E$, $\underset{t\in[s,T]}{\text{sup }}|\psi_i(t,X_t)|^2$ belongs to $L^1$ under $\mathbbm{P}^{s,x}$;
		\item $\psi_i\in\mathcal{L}^2_X$.
	\end{itemize}
\end{proposition}
\begin{proof}
The first item follows from the fact that, for any $(s,x)\in[0,T]\times E$,
 $\langle M[\psi_i]^{s,x}\rangle=\int_s^{\cdot\vee s}\mathfrak{G}^{\psi_i}(\psi_i)(r,X_r)dV_r$ (see Proposition \ref{P321}), and the fact that 
$\mathfrak{G}^{\psi_i}(\psi_i)$ is bounded. Concerning the second item, for any $(s,x)\in[0,T]\times E$, the martingale problem gives  $\psi_i(\cdot,X)=\psi_i(s,x)+\int_s^{\cdot}a(\psi_i)(r,X_r)dV_r+M[\psi_i]^{s,x}$, see Definition \ref{MartingaleProblem}. By Jensen's inequality, we have $\underset{t\in[s,T]}{\text{sup }}|\psi_i(t,X_t)|^2\leq C(\psi^2_i(s,x)+\int_s^{T}a^2(\psi_i)(r,X_r)dV_r+\underset{t\in[s,T]}{\text{sup }}(M[\psi_i]_t^{s,x})^2)$ for some $C>0$. It is therefore $L^1$ since $a(\psi_i)\in\mathcal{L}^2_X$ and $M[\psi_i]^{s,x}\in\mathcal{H}^2$. The last item is a direct consequence of the second one.
\end{proof}

\begin{definition}\label{MarkovPDE}
Let us consider some $g\in\mathcal{B}(E,\mathbbm{R})$ and 
\\
$f\in\mathcal{B}([0,T]\times E\times\mathbbm{R}\times\mathbbm{R},\mathbbm{R}^d)$.

We will call \textbf{Pseudo-Partial Differential Equation} 
related to $(f,g)$ (in short $Pseudo-PDE(f,g)$) the following equation with final condition:
\begin{equation}\label{PDE}
\left\{
\begin{array}{rccc}
 a(u) + f\left(\cdot,\cdot,u,\Gamma^{\psi}(u)\right)&=&0& \text{ on } [0,T]\times E   \\
 u(T,\cdot)&=&g.& 
\end{array}\right.
\end{equation}
We will say that $u$ is a \textbf{classical solution} of $Pseudo-PDE(f,g)$ if 
  $u, u\psi_i, i\in [\![ 1;d]\!]$  belong to $\mathcal{D}(a)$
and  if $u$   satisfies \eqref{PDE}.
\end{definition}

The connection between a Markovian BSDE
 and
a $Pseudo-PDE(f,g)$, will be possible under a hypothesis 
on some generalized moments on $X$, and some growth conditions
 on the functions $(f,g)$. Those will be related to 
two fixed functions $\zeta,\eta \in\mathcal{B}(E,\mathbbm{R}_+)$.
\begin{hypothesis}\label{HypMom}
The canonical Markov class will be said \textbf{to satisfy}  $H^{mom}(\zeta,\eta)$ if 
\begin{enumerate}
\item for any $(s,x)\in[0,T]\times E$, $\mathbbm{E}^{s,x}[\zeta^2(X_T)]$ is finite;
\item for any $(s,x)\in[0,T]\times E$, $\mathbbm{E}^{s,x}\left[\int_0^T\eta^2(X_r)dV_r\right]$ is finite.
\end{enumerate}
\end{hypothesis}
Until the end of this section, we assume that some $\zeta,\eta$ are given and that the canonical Markov class satisfies $H^{mom}(\zeta,\eta)$.

\begin{hypothesis}\label{Hpq}
A couple $(f,g)$ of functions  $f\in\mathcal{B}([0,T]\times E\times\mathbbm{R}\times\mathbbm{R}^d,\mathbbm{R})$ and $g\in\mathcal{B}(E,\mathbbm{R})$ will be said \textbf{to satisfy}   $H^{lip}(\zeta,\eta)$ if
there exist positive  constants $K^Y,K^Z,C,C'$ such that
\begin{enumerate}
\item $\forall x: \quad |g(x)|\leq C(1+\zeta(x))$,
\item $\forall (t,x):\quad  |f(t,x,0,0)|\leq C'(1+\eta(x))$,
\item $\forall (t,x,y,y',z,z'):\quad  |f(t,x,y,z)-f(t,x,y',z')|\leq K^Y|y-y'|+K^Z\|z-z'\|$.
\end{enumerate}
$(f,g)$ will be said \textbf{to satisfy} $H^{growth}(\zeta,\eta)$ if
 the following more general assumption holds. There exist positive  constants $C,C'$ such that
	\begin{enumerate}
		\item $\forall x: \quad |g(x)|\leq C(1+\zeta(x))$;
		\item $\forall (t,x,y,z):\quad  |f(t,x,y,z)|\leq C'(1+\eta(x)+|y|+\|z\|)$.
	\end{enumerate}
\end{hypothesis}

\begin{remark}\label{MarkovBSDEsol2}
We fix for now a couple $(f,g)$ satisfying $H^{lip}(\zeta,\eta)$.
For any $(s,x)\in[0,T]\times E$, in the stochastic basis 	
$\left(\Omega,\mathcal{F}^{s,x},(\mathcal{F}^{s,x}_t)_{t\in[0,T]},\mathbbm{P}^{s,x}\right)$ and setting $\hat V:= V$, the triplet 
$\xi:=g(X_T)$, $\hat{f}:(t,\omega,y,z)\longmapsto f(t,X_t(\omega),y,z)$, $\hat{M}:=M[\psi]^{s,x}$ satisfies Hypothesis \ref{HypBSDE}.
\end{remark}

With the equation $Pseudo-PDE(f,g)$, we will associate the following family of BSDEs indexed by $(s,x)\in[0,T]\times E$,  driven by a c\`adl\`ag martingale.

\begin{notation} \label{N55bis}
	For any $(s,x)\in[0,T]\times E$, we consider in the stochastic 
basis $\left(\Omega,\mathcal{F}^{s,x},(\mathcal{F}^{s,x}_t)_{t\in[0,T]},\mathbbm{P}^{s,x}\right)$ and on the interval $[0,T]$ the 
$BSDE(\xi,\hat{f}, V,\hat{M})$ where 
$\xi = g(X_T)$, $\hat f: (t,\omega,y,z)\longmapsto f(t,X_t(\omega),y,z)$,
$\hat M = M[\psi]^{s,x})$.  	\\
	From now on that BSDE will be denoted $BSDE^{s,x}(f,g)$ and its unique solution (see Theorem \ref{uniquenessBSDE} and Remark \ref{MarkovBSDEsol2}) will be denoted $(Y^{s,x},M^{s,x})$.
\end{notation}
If $H^{lip}(\zeta,\eta)$ is fulfilled by $(f,g)$, then
$(Y^{s,x},M^{s,x})$ is therefore the unique couple in $\mathcal{L}^2(dV\otimes d\mathbbm{P}^{s,x})\times \mathcal{H}^2_0$ satisfying 
\begin{equation}\label{BSDE}
Y^{s,x}_{\cdot} = g(X_T) + \int_{\cdot}^T f\left(r,X_r,Y^{s,x}_r,\frac{d\langle M^{s,x},M[\psi]^{s,x}\rangle}{dV}(r)\right)dV_r  -(M^{s,x}_T - M^{s,x}_{\cdot}).
\end{equation}

\begin{remark}\label{MarkovBSDEsol}
Even if  the underlying process $X$ admits no generalized  moments,  given 
  a couple  $(f,g)$ such that $f(\cdot,\cdot,0,0)$ and $g$ are bounded, the considerations of this section still apply. In particular  the connections that we will establish between the
 $BSDE^{s,x}(f,g)$ and the corresponding $Pseudo-PDE(f,g)$ still take place.
\end{remark}

Our main contribution consists in  illustrating the precise link
 between  the solutions of equations $BSDE^{s,x}(f,g)$ and those of
  $Pseudo-PDE(f,g)$.
In particular we will emphasize that a solution of $BSDE^{s,x}(f,g)$
 produces a solution of $Pseudo-PDE(f,g)$  and 
reciprocally.

We now introduce a probabilistic notion of solution for $Pseudo-PDE(f,g)$.
\begin{definition}\label{D417}
	A Borel function $u: [0,T] \times E \rightarrow {\mathbbm R}$
	will be said 
	to be a {\bf martingale solution} of $Pseudo-PDE(f,g)$ if 
	$u\in\mathcal{D}(\mathfrak{a})$  and
	\begin{equation}\label{PDEextended}
	\left\{\begin{array}{rcl}
	\mathfrak{a}(u)&=& -f(\cdot,\cdot,u,\mathfrak{G}^{\psi}(u))\\
	u(T,\cdot)&=&g.
	\end{array}\right.
	\end{equation}
\end{definition}

\begin{remark} \label{R417}
	The first equation of \eqref{PDEextended} holds in $L^0_X$, hence up to a zero potential set. The second one is a pointwise equality.
\end{remark}
The following lemma was the object  of Lemma 5.13 in \cite{paper1preprint}.
\begin{lemma}\label{ModifImpliesdV}
	Let $V$ be a non-decreasing function. If two measurable processes are $\mathbbm{P}$-modifications of each other, then they are also equal $dV\otimes d\mathbbm{P}$ a.e.
\end{lemma}
\begin{proposition} \label{MartImpliesBSDE}
	Let $(f,g)$ satisfy $H^{growth}(\zeta,\eta)$.
	Let $u$ be a martingale solution of $Pseudo-PDE(f,g)$. Then for any $(s,x)\in[0,T]\times E$, the couple of processes
	\begin{equation}
	\left(u(t,X_t),\quad u(t,X_t)-u(s,x)+\int_s^tf(\cdot,\cdot,u,\mathfrak{G}^{\psi}(u))(r,X_r)dV_r\right)_{t \in [s,T]}
	\end{equation}
	has 
a $\mathbbm{P}^{s,x}$-version  which is a solution on $[s,T]$ of $BSDE^{s,x}(f,g)$,
see Remark \ref{BSDESmallInt}.

	Moreover, $u\in\mathcal{L}^2_X$.
\end{proposition}

\begin{proof}
	Let $u\in\mathcal{D}(\mathfrak{a})$ be a solution of
 \eqref{PDEextended} and let 
	$(s,x)\in[0,T]\times E$ be fixed. By Definition
 \ref{domainextended} and 
Remark	\ref{BSDESmallInt}, the process $u(\cdot,X_{\cdot})$  under $\mathbbm{P}^{s,x}$ admits 
	a c\`adl\`ag modification $U^{s,x}$  on $[s,T]$,
	which is a special semimartingale with decomposition 
	\begin{equation} \label{E531}
	\begin{array}{rcl}
	U^{s,x} &=& u(s,x) + \int_s^{\cdot} \mathfrak{a}(u)(r,X_r)dV_r + M[u]^{s,x} \\
	&=&   u(s,x) - \int_s^{\cdot} f\left(r,X_r,u(r,X_r),\mathfrak{G}^{\psi}(u)(r,X_r)\right)dV_r + M[u]^{s,x} \\
	&=&	u(s,x) - \int_s^{\cdot} f\left(r,X_r,U^{s,x}_r,\frac{d\langle  M[u]^{s,x},M[\psi]^{s,x}\rangle}{dV}(r)\right)dV_r + M[u]^{s,x},
	\end{array}
	\end{equation}  
	where the third equality of \eqref{E531} comes from Lemma \ref{ModifImpliesdV} and  Proposition \ref{P321}.
	Moreover since $u(T,\cdot)=g$, then  $U^{s,x}_T=u(T,X_T)=g(X_T)$ a.s. so 
	the couple  $(U^{s,x}, M[u]^{s,x})$ 
	satisfies the following equation on $[s,T]$ (with respect to $\mathbbm P^{s,x}$):
	\begin{equation} \label{EBSDEweaker}
	U^{s,x}_{\cdot} = g(X_T)+\int_{\cdot}^Tf\left(r,X_r,U^{s,x}_r,\frac{d\langle M[u]^{s,x},M[\psi]^{s,x}\rangle}{dV}(r)\right)dV_r - (M[u]^{s,x}_T-M[u]^{s,x}_{\cdot}).
	\end{equation}
	$M[u]^{s,x}$   (introduced at Definition \ref{extended}) belongs to
	$\mathcal{H}^2_0$ but we do not have a priori information  on the square
	integrability of $U^{s,x}$. However we know that  $M[u]^{s,x}$ is 
	equal to zero at time $s$, and that $U^{s,x}_s$ is deterministic so 
	square integrable. We can therefore apply  Lemma \ref{LED+Pext} which implies that  $(U^{s,x},M[u]^{s,x})$ 
	solves $BSDE^{s,x}(f,g)$ on $[s,T]$. In particular, $U^{s,x}$ belongs to $\mathcal{L}^2(dV\otimes d\mathbbm{P}^{s,x})$ for every $(s,x)$, so by Lemma  \ref{ModifImpliesdV} and Definition \ref{topo}, $u\in\mathcal{L}^2_X$.
\end{proof}

\subsection{Decoupled mild solutions of Pseudo-PDEs}\label{S4}

In this section we introduce an analytical notion of  solution of 
our $Pseudo-PDE(f,g)$, that we will denominate {\it decoupled mild}, 
taking inspiration from  the mild solutions of partial differential equation. 
That notion will be shown to be equivalent to the one of martingale solution introduced in Definition \ref{D417}.
	Let $P = (P_{s,t})$ denote the transition kernel of the canonical Markov
  class, see Definition 3.4 in \cite{paperAF} and also Notation \ref{N513}.


Our notion of decoupled mild solution relies on the fact that the equation 
$a(u) + f\left(\cdot,\cdot,u,\Gamma^{\psi}(u)\right)=0$ 
can be naturally decoupled into 
\begin{equation} \label{E62}
\left\{
\begin{array}{ccl}
a(u) &=& - f(\cdot,\cdot,u,v)\\
v_i &=&  \Gamma^{\psi_i}(u),\quad i\in[\![1;d]\!].
\end{array}\right.
\end{equation}
Then, by definition of the carr\'e du champ operator (see Definition \ref{CarreDuChamp}), we formally have 
  $a(u\psi_i)=\Gamma^{\psi_i}(u)+ua(\psi_i)+\psi_ia(u), i\in[\![ 1;d]\!]$. So 
the  system of equations \eqref{E62} can be rewritten as
\begin{equation}
\left\{
\begin{array}{ccl}
a(u) &=& - f(\cdot,\cdot,u,v)\\
a(u\psi_i) &=& v_i +ua(\psi_i)- \psi_if(\cdot,\cdot,u,v),\quad i\in[\![1;d]\!].
\end{array}\right.
\end{equation}
Inspired by the usual notions of mild solution, this naturally brings
 us to the following definition of a (decoupled) mild solution.

\begin{definition}\label{mildsoluv}
	Assume that $(f,g)$ satisfies $H^{growth}(\zeta,\eta)$. Let 
	\\
	$u\in\mathcal{B}([0,T]\times E,\mathbbm{R})$ and $v\in\mathcal{B}([0,T]\times E,\mathbbm{R}^d)$.
	\begin{enumerate}
		\item  $(u,v)$ is a {\bf solution
			of the identification problem determined by $(f,g)$}
		or simply {\bf solution of} 
		$IP(f,g)$ if 
		$u,v_1,\cdots,v_d$  belong to $\mathcal{L}^2_X$ and if for every $(s,x)\in[0,T]\times E$,
		\begin{equation}\label{MildEq}
		\left\{
		\begin{array}{rcl}
		u(s,x)&=&P_{s,T}[g](x)+\int_s^TP_{s,r}\left[f\left(\cdot,\cdot,u,v\right)(r,\cdot)\right](x)dV_r\\
		u\psi_1(s,x) &=&P_{s,T}[g\psi_1(T,\cdot)](x) -\int_s^TP_{s,r}\left[\left(v_1+ua(\psi_1)-\psi_1f\left(\cdot,\cdot,u,v\right)\right)(r,\cdot)\right](x)dV_r\\
		&\cdots&\\
		u\psi_d(s,x) &=&P_{s,T}[g\psi_d(T,\cdot)](x) -\int_s^TP_{s,r}\left[\left(v_d+ua(\psi_d)-\psi_df\left(\cdot,\cdot,u,v\right)\right)(r,\cdot)\right](x)dV_r.
		\end{array}\right.
		\end{equation}
		\item  $u$ is a {\bf decoupled mild solution}
		of $Pseudo-PDE(f,g)$ if there exists a function $v$
		such that  $(u,v)$ is a solution 
		of $IP(f,g)$.
	\end{enumerate}
\end{definition}

The following lemma is very close to 
 Lemma 3.5 in \cite{paper2} 
and the arguments for the proof are similar.
\begin{lemma}\label{LemmaMild}
	Let $u,v_1,\cdots,v_d\in\mathcal{L}^2_X$, let $(f,g)$ be a couple satisfying $H^{growth}(\zeta,\eta)$ and let $\psi_1,\cdots,\psi_d$ satisfy Hypothesis \ref{HypPhi}.
	Then $f\left(\cdot,\cdot,u,v\right)$ belongs to $\mathcal{L}_X^2$ 
	and for every $i\in[\![ 1;d]\!]$, $\psi_if\left(\cdot,\cdot,u,v\right)$,  and $ua(\psi_i)$,  belong to $\mathcal{L}^1_X$. For any $(s,x)\in[0,T]\times E$, $i\in[\![ 1;d]\!]$, $g(X_T)\psi_i(T,X_T)$ belongs to $L^1$ under $\mathbbm{P}^{s,x}$. In particular, all terms  in \eqref{MildEq} make sense.
\end{lemma}


\begin{proposition}\label{MartingaleImpliesMild} 
	Let $(f,g)$ satisfy $H^{growth}(\zeta,\eta)$.
	Let $u$ be a martingale solution of $Pseudo-PDE(f,g)$, then $(u,\mathfrak{G}^{\psi}(u))$ is a
	solution of $IP(f,g)$ and in particular, $u$ is a decoupled mild solution of $Pseudo-PDE(f,g)$.
\end{proposition}

\begin{proof}
	Let $u$ be a martingale solution of $Pseudo-PDE(f,g)$. By Proposition \ref{MartImpliesBSDE}, $u\in\mathcal{L}^2_X$. Taking into account Definition \ref{extended}, for every $(s,x)$,  $M[u]^{s,x}\in\mathcal{H}^2_0$ under $\mathbbm{P}^{s,x}$.
	So by Lemma \ref{ZinL2},  for any $i\in[\![ 1;d]\!]$,  $\frac{d\langle M[u]^{s,x},M[\psi^i]^{s,x}\rangle}{dV}$ 
	belongs to  $\mathcal{L}^2(dV\otimes d\mathbbm{P}^{s,x})$. By use of 
Proposition \ref{P321}, this means that $\mathfrak{G}^{\psi_i}(u)\in\mathcal{L}^2_X$ for every $i$.    
	By Lemma \ref{LemmaMild}, it follows that $f\left(\cdot,\cdot,u,\mathfrak{G}^{\psi}(u)\right)$ belongs to
	$\mathcal{L}_X^2$ and so for any $i\in[\![ 1;d]\!]$, $\psi_if\left(\cdot,\cdot,u,\mathfrak{G}^{\psi}(u)\right)$ and $ua(
	\psi_i)$,  belong to $\mathcal{L}^1_X$.
	
	Let $(s,x)\in[0,T]\times E$. 
Below we demonstrate that
	\begin{equation}\label{MildEqAux}
	\left\{
	\begin{array}{rcl}
	u(s,x)&=&P_{s,T}[g](x)+\int_s^TP_{s,r}\left[f\left(\cdot,\cdot,u,\mathfrak{G}^{\psi}(u)\right)(r,\cdot)\right](x)dV_r\\
	u\psi_1(s,x) &=&P_{s,T}[g\psi_1(T,\cdot)](x) -\int_s^TP_{s,r}\left[\left(\mathfrak{G}(u,\psi_1)+ua(\psi_1)-\psi_1f\left(\cdot,\cdot,u,\mathfrak{G}^{\psi}(u)\right)\right)(r,\cdot)\right](x)dV_r\\
	&\cdots&\\
	u\psi_d(s,x) &=&P_{s,T}[g\psi_d(T,\cdot)](x) -\int_s^TP_{s,r}\left[\left(\mathfrak{G}(u,\psi_d)+ua(\psi_d)-\psi_df\left(\cdot,\cdot,u,\mathfrak{G}^{\psi}(u)\right)\right)(r,\cdot)\right](x)dV_r.
	\end{array}\right.
	\end{equation}

We refer now to the probability  $\mathbbm{P}^{s,x}$: 
	by Definitions \ref{domainextended}, \ref{extended} and \ref{D417}, 
the process $u(\cdot,X_{\cdot})$ admits a modification  $U^{s,x}$ being a special semimartingale with decomposition
	\begin{equation}\label{decompoU}
	U^{s,x}=u(s,x)-\int_s^{\cdot}f\left(\cdot,\cdot,u,\mathfrak{G}^{\psi}(u)\right)(r,X_r)dV_r +M[u]^{s,x},
	\end{equation}
	and $M[u]^{s,x}\in\mathcal{H}^2_0$.
	
	Definition \ref{D417} also states that $u(T,\cdot)=g$, so
	\begin{equation}
	u(s,x)=g(X_T)+\int_s^Tf\left(\cdot,\cdot,u,\mathfrak{G}^{\psi}(u)\right)(r,X_r)dV_r -M[u]^{s,x}_T\text{ a.s.}
	\end{equation}
	By Fubini's theorem we deduce that
	\begin{equation}
	\begin{array}{rcl}
	u(s,x) &=&\mathbbm{E}^{s,x}\left[g(X_T) +\int_s^T f\left(\cdot,\cdot,u,\mathfrak{G}^{\psi}(u)\right)(r,X_r)dV_r\right]\\
	&=&P_{s,T}[g](x)+\int_s^TP_{s,r}\left[f\left(r,\cdot,u(r,\cdot),\mathfrak{G}^{\psi}(u)(r,\cdot)\right)\right](x)dV_r.
	\end{array}
	\end{equation}
	We now fix $i\in[\![ 1;d]\!]$. By integration by parts, taking \eqref{decompoU} and Definition \ref{MartingaleProblem} into account, we obtain
	\begin{equation}
	\begin{array}{rcl}
	d(U^{s,x}_t\psi_i(t,X_t))&=&-\psi_i(t,X_t)f\left(\cdot,\cdot,u,\mathfrak{G}^{\psi}(u)\right)(t,X_t)dV_t+\psi_i(t^-,X_{t^-})dM[u]^{s,x}_t\\
	&& +U^{s,x}_ta(\psi_i)(t,X_t)dV_t+U^{s,x}_{t^-}dM[\psi_i]^{s,x}_t+d[M[u]^{s,x},M[\psi_i]^{s,x}]_t,
	\end{array}	
	\end{equation} 
	Integrating between $s$ and $T$, 
	\begin{equation}
	\begin{array}{rcl}\label{Eq322}
	&&u\psi_i(s,x)\\
	&=& g(X_T)\psi_i(T,X_T)+\int_s^T\psi_i(t,X_t)f\left(\cdot,\cdot,u,\mathfrak{G}^{\psi}(u)\right)(r,X_r)dV_r-\int_s^T\psi_i(r^-,X_{r^-})dM[u]^{s,x}_r\\
	&& -\int_s^TU^{s,x}_ta(\psi_i)(r,X_r)dV_r-\int_s^TU^{s,x}_{r^-}dM[\psi_i]^{s,x}_r-[M[u]^{s,x},M[\psi_i]^{s,x}]_T\\
	&=& g(X_T)\psi_i(T,X_T)-\int_s^T\left(ua(\psi_i)-\psi_if\left(\cdot,\cdot,u,\mathfrak{G}^{\psi}(u)\right)\right)(r,X_r)dV_r-\int_s^T\psi_i(r^-,X_{r^-})dM[u]^{s,x}_r\\
	&& -\int_s^TU^{s,x}_{r^-}dM[\psi_i]^{s,x}_r-[M[u]^{s,x},M[\psi_i]^{s,x}]_T,\\
	\end{array}
	\end{equation}
	thanks to Lemma \ref{ModifImpliesdV}.
	
	By Proposition \ref{P321}, $\langle M[\psi_i]^{s,x}\rangle=\int_s^{\cdot\vee s}\mathfrak{G}^{\psi_i}(\psi_i)(r,X_r)dV_r$. 
	So the angular bracket of $\int_s^{\cdot}U^{s,x}_{r^-}dM[\psi_i]^{s,x}_r$ at time $T$ is equal to  $\int_s^Tu^2\mathfrak{G}^{\psi_i}(\psi_i)(r,X_r)dV_r$ which is an integrable r.v. since $\mathfrak{G}^{\psi_i}(\psi_i)$ is bounded and $u\in\mathcal{L}^2_X$. Therefore $\int_s^{\cdot}U^{s,x}_{r^-}dM[\psi_i]^{s,x}_r$  is a square integrable martingale.
	
	Then, by Hypothesis \ref{HypPhi} and Proposition \ref{R57}, 
 $\underset{t\in[s,T]}{\text{sup }}|\psi_i(t,X_t)|^2\in L^1$, and by Definition \ref{extended}, $M[u]^{s,x} \in \mathcal{H}^2$ so by Lemma 3.17 in 
\cite{paper1preprint}, $\int_s^{\cdot}\psi_i(r^-,X_{r^-})dM[u]^{s,x}_r$ is a martingale.
	
We can now perform the expectation in \eqref{Eq322}, to get  
	\begin{equation}
	\begin{array}{rcl}
	&&u\psi_i(s,x)\\
	&=&\mathbbm{E}^{s,x}\left[g(X_T)\psi_i(T,X_T)-\int_s^T\left(ua(\psi_i)-\psi_if\left(\cdot,\cdot,u,\mathfrak{G}^{\psi}(u)\right)\right)(r,X_r)dV_r -[M[u]^{s,x},M[\psi_i]^{s,x}]_T\right]\\
	&=&\mathbbm{E}^{s,x}\left[g(X_T)\psi_i(T,X_T)-\int_s^T\left(ua(\psi_i)+\mathfrak{G}^{\psi_i}(u)-\psi_if\left(\cdot,\cdot,u,\mathfrak{G}^{\psi}(u)\right)\right)(r,X_r)dV_r\right],
	\end{array}
	\end{equation}
	since $u$ and $\psi_i$ belong to $\mathcal{D}(\mathfrak{a})$.
Indeed  the second equality follows from the fact
  $[M[u]^{s,x},M[\psi_i]^{s,x}] - \langle M[u]^{s,x},M[\psi_i]^{s,x} \rangle$
is a martingale  and Proposition \ref{P321}.

 Since we have assumed that $u\in\mathcal{L}^2_X$,
	Lemma \ref{LemmaMild} says that $f\left(\cdot,\cdot,u,\mathfrak{G}^{\psi}(u)\right)\in\mathcal{L}^2_X$, Hypothesis \ref{HypPhi} implies that $\psi_i$ and $a(\psi_i)$ are in $\mathcal{L}^2_X$, so all terms in the integral
inside the expectation in the third line belong to $\mathcal{L}^1_X$.
 We can therefore apply Fubini's theorem to get 
	\begin{equation}
	u\psi_i(s,x)=P_{s,T}[g\psi_i(T,\cdot)](x)-\int_s^TP_{s,r}\left[\left(ua(\psi_i)+\mathfrak{G}^{\psi_i}(u)-\psi_if\left(\cdot,\cdot,u,\mathfrak{G}^{\psi}(u)\right)\right)(r,\cdot)\right](x)dV_r.
	\end{equation}
	This concludes the proof.
\end{proof}

Proposition \ref{MartingaleImpliesMild} admits a converse implication.

\begin{proposition}\label{MildImpliesMartingale} 
	Let $(f,g)$ satisfy $H^{growth}(\zeta,\eta)$, then
	every decoupled mild solution of $Pseudo-PDE(f,g)$ is  a martingale solution.   Moreover, if $(u,v)$ solves $IP(f,g)$, then $v=\mathfrak{G}^{\psi}(u)$, up to zero potential sets.
\end{proposition}

\begin{proof} 
	Let $u$ and $v_i$, $i\in[\![ 1;d]\!]$  in $\mathcal{L}^2_X$ satisfy \eqref{MildEq}.  We observe
	that the first line of \eqref{MildEq} with $s=T$, implies that $u(T,\cdot)=g$. 

	Let $(s,x)\in[0,T]\times E$ be fixed. We will now work under the probability $\mathbbm{P}^{s,x}$. On $[s,T]$, we set $U:=u(\cdot,X)$ and $N:=u(\cdot,X)-u(s,x)+\int_s^{\cdot}f(r,X_r,u(r,X_r),v(r,X_r))dV_r$.
	
	For some $t\in[s,T]$, we combine the first line of \eqref{MildEq} applied with $(s,x)=(t,X_t)$ 
and the Markov property, see e.g. (3.4) in \cite{paperAF}.
 Since $f\left(\cdot,\cdot,u,v\right)$ belongs to $\mathcal{L}^2_X$ (see Lemma \ref{LemmaMild}) we a.s. have that
	\begin{equation}
	\begin{array}{rcl}
	U_t &=& u(t,X_t) \\
	&=& P_{t,T}[g](X_t) + \int_t^TP_{t,r}\left[f\left(r,\cdot,u(r,\cdot),v(r,\cdot)\right)\right](X_t)dV_r\\
	&=& \mathbbm{E}^{t,X_t}\left[g(X_T)+\int_t^T f(r,X_r,u(r,X_r),v(r,X_r))dV_r\right] \\
	&=& \mathbbm{E}^{s,x}\left[g(X_T)+\int_t^T f(r,X_r,u(r,X_r),v(r,X_r))dV_r|\mathcal{F}_t\right],
	\end{array}
	\end{equation}
	so
	$N_t=\mathbbm{E}^{s,x}\left[g(X_T)+\int_s^T f(r,X_r,u(r,X_r),v(r,X_r))dV_r|\mathcal{F}_t\right]-u(s,x)$ a.s. hence $N$  is a martingale. Let $N^{s,x}$ denote its c\`adl\`ag version which we extend on   $[0,s]$ with the value 0. Then 
	\begin{equation}\label{E5000}
	U^{s,x} := u(s,x)- \int_s^{\cdot} f(r,X_r,u(r,X_r),v(r,X_r))dV_r + N^{s,x},
	\end{equation}
indexed on $[s,T]$ is a c\`adl\`ag  version of $U$. Proceeding as in
the proof of Proposition 3.8 in \cite{paper2}, we can show that
 $N^{s,x}$ is a square integrable martingale.
%
%
The process
 $\left(u(\cdot,X_{\cdot})-u(s,x)+\int_s^{\cdot}f(r,X_r,u(r,X_r),v(r,X_r))dV_r\right)\mathds{1}_{[s,T]}$
therefore admits for any $(s,x)$ a $\mathbbm{P}^{s,x}$-modification in $\mathcal{H}^2_0$ . By  Definitions \ref{domainextended}, \ref{extended} this means that  $u\in\mathcal{D}(\mathfrak{a})$, 
	$\mathfrak{a}(u)=-f(\cdot,\cdot,u,v)$ and 
 for any $(s,x)\in[0,T]\times E$, $M[u]^{s,x}=N^{s,x}$ $P^{s,x}$-a.s.
	
	We are left to show  $\mathfrak{G}^{\psi}(u)=v$, up to zero potential
 sets, hence that for every $(s,x)\in[0,T]\times E$ and $i\in[\![ 1;d]\!]$,  
	\begin{equation} \label{P3210}
	\langle M^{s,x}[u],M^{s,x}[\psi_i]\rangle = \int_s^{\cdot\vee s}v_i(r,X_r)dV_r,\quad\text{a.s.}
\end{equation} 
	Let $(s,x)\in[0,T]\times E,$ 
and  $i\in[\![ 1;d]\!]$ .
 Combining the 
	$(i+1)$-th line of \eqref{MildEq} applied to $(s,x)=(t,X_t)$ and the Markov property and the Markov property (see e.g. (3.4) in \cite{paperAF}),
	taking into account the fact that
	all terms belong to $\mathcal{L}^1_X$ (see Lemma \ref{LemmaMild}, Hypothesis \ref{HypPhi}) we a.s. have
	\begin{equation} \label{E522}
	\begin{array}{rcl}
	u\psi_i(t,X_t) &=& P_{t,T}[g\psi_i(T,\cdot)](X_t) - \int_t^TP_{t,r}\left[\left(v_i+ua(\psi_i)-\psi_if\left(\cdot,\cdot,u,v\right)\right)(r,\cdot)\right](X_t)dV_r\\
	&=& \mathbbm{E}^{t,X_t}\left[g(X_T)\psi_i(T,X_T)-\int_t^T\left(v_i+ua(\psi_i)-\psi_if\left(\cdot,\cdot,u,v\right)\right)(r,X_r)dV_r\right] \\
	&=& \mathbbm{E}^{s,x}\left[g(X_T)\psi_i(T,X_T)-\int_t^T\left(v_i+ua(\psi_i)-\psi_if\left(\cdot,\cdot,u,v\right)\right)(r,X_r)dV_r|\mathcal{F}_t\right].
	\end{array}
	\end{equation}
Setting, for $t \in [s,T]$,
	$N^i_t:=u\psi_i(t,X_t)
-\int_s^t(v_i+ua(\psi)^i-\psi_if(\cdot,\cdot,u,v))(r,X_r)dV_r$,
from \eqref{E522}
 we deduce that, for any $t\in[s,T]$,
	$$N^i_t=\mathbbm{E}^{s,x}\left[g(X_T)\psi_i(T,X_T)-\int_s^T \left(v_i+ua(\psi_i)-\psi_if\left(\cdot,\cdot,u,v\right)\right)(r,X_r)dV_r
\middle|\mathcal{F}_t\right]
$$
 a.s.
	So $N^i$  is a martingale. Let $N^{i,s,x}$ denote its c\`adl\`ag $\mathbbm{P}^{s,x}$-modification. The process 
\begin{equation} \label{EV1}	
\int_s^{\cdot}\left(v_i+ua(\psi_i)-\psi_if\left(\cdot,\cdot,u,v\right)\right)(r,X_r)dV_r + N^{i,s,x},
\end{equation}
 is  a c\`adl\`ag $\mathbbm{P}^{s,x}$-version of $u\psi_i(\cdot,X)$ on $[s,T]$.
 But we have by \eqref{E5000}, that   
	\\
	$U^{s,x}=u(s,x)- \int_s^{\cdot} f(r,X_r,u(r,X_r),v(r,X_r))dV_r + N^{s,x}$ is a version of $u(\cdot,X)$, hence by integration by parts on  $U^{s,x}\psi_i(\cdot,X_{\cdot})$  that 
	\begin{equation} \label{E523}
	\begin{array}{l}
	u\psi_i(s,x)+\int_s^{\cdot}U^{s,x}_ra(\psi_i)(r,X_r)dV_r+\int_s^{\cdot}U^{s,x}_{r^-}dM^{s,x}[\psi_i]_r\\
	-\int_s^{\cdot}\psi_if(\cdot,\cdot,u,v)(r,X_r)dV_r+\int_s^{\cdot}\psi_i(r^-,X_{r^-})dM^{s,x}[u]_r+[M^{s,x}[u],M^{s,x}[\psi_i]]
	\end{array}
	\end{equation}
	is another c\`adl\`ag  semimartingale which
 is a $\mathbbm{P}^{s,x}$-version of $u\psi_i(\cdot,X)$ on $[s,T]$.
Now \eqref{E523} equals 
\begin{equation} \label{EMV}
{\mathcal M^i} + {\mathcal V^i},
\end{equation}
where 
\begin{eqnarray*} 
{\mathcal M}^i_t &=& u \psi_i(s,x) + \int_s^{t}U^{s,x}_{r^-}dM^{s,x}[\psi_i]_r+\int_s^{t}\psi_i(r^-,X_{r^-})dM^{s,x}[u]_r \\
&+&([M^{s,x}[u],M^{s,x}[\psi_i]]_t-\langle M^{s,x}[u],M^{s,x}[\psi_i]\rangle_t,
\end{eqnarray*}
is a local martingale and
$$   {\mathcal V}^i_t =  \langle M^{s,x}[u],M^{s,x}[\psi_i]\rangle_t +\int_s^{t}U^{s,x}_ra(\psi_i)(r,X_r)dV_r
	-\int_s^{t}\psi_if(\cdot,\cdot,u,v)(r,X_r)dV_r,$$
is a predictable process with bounded variation vanishing at zero. Now \eqref{EMV} and \eqref{EV1} are two c\`adl\`ag versions
of $u\psi_i(\cdot,X)$ on $[s,T]$. 

        By the uniqueness of the decomposition of a special semimartingale
and using Lemma \ref{ModifImpliesdV}
we get 
\begin{eqnarray*}
\int_s^{\cdot}  (v_i &+& ua(\psi_i)-\psi_if (\cdot,\cdot,u,v))(r,X_r)dV_r \\
&=&	
	\langle M^{s,x}[u],M^{s,x}[\psi_i]\rangle +\int_s^{\cdot}ua(\psi_i)(r,X_r)dV_r
	-\int_s^{\cdot}\psi_if(\cdot,\cdot,u,v)(r,X_r)dV_r.
\end{eqnarray*}
 This yields $\langle M^{s,x}[u],M^{s,x}[\psi_i]\rangle=\int_s^{\cdot\vee s}v_i(r,X_r)dV_r$, which implies \eqref{P3210}.
	
\end{proof}

\begin{proposition}\label{CoroClassic}
	Let $(f,g)$ satisfy $H^{growth}(\zeta,\eta)$.
	A classical solution of $Pseudo-PDE(f,g)$   is a decoupled mild solution.

	Conversely, a decoupled mild solution of $Pseudo-PDE(f,g)$  
	belonging to $\mathcal{D}(\Gamma^{\psi})$ is a classical solution of $Pseudo-PDE(f,g)$ up to a zero-potential set, meaning that it satisfies the first equality of \eqref{PDE} up to a set of zero potential. 
\end{proposition}

\begin{proof} 
	Let $u$ be a classical solution of $Pseudo-PDE(f,g)$. Definition \ref{MarkovPDE}   and Corollary \ref{RExtendedClassical} imply 
	that $u(T,\cdot)=g$, and the equalities up to zero potential sets
	\begin{equation}
	\mathfrak{a}(u) = a(u)
	=-f(\cdot,\cdot,u,\Gamma^{\psi}(u))
	=-f(\cdot,\cdot,u,\mathfrak{G}^{\psi}(u)),
	\end{equation}
	which shows that $u$ is a martingale solution and by Proposition \ref{MartingaleImpliesMild} it is also a decoupled mild solution.
	
	Similarly, the second statement follows by Proposition \ref{MildImpliesMartingale}, Definition \ref{D417}, and again Corollary \ref{RExtendedClassical}.
\end{proof}

\subsection{Existence and uniqueness of a decoupled mild solution}\label{mart-mild}
In this subsection, the positive functions $\zeta,\eta$ and
the functions $(f,g)$ appearing in $Pseudo-PDE(f,g)$ are  fixed. 
We still  assume that the canonical Markov class satisfies $H^{mom}(\zeta,\eta)$. 

Theorem \ref{Defuv} below can be proved
using arguments which are very close to those developed 
in the proof of Theorem 5.15 in \cite{paper1preprint}.
For the convenience of the reader, we postpone the adapted proof to 
Appendix \ref{B}.

Let $(Y^{s,x},M^{s,x})$ be for any $(s,x)\in[0,T]\times E$   the unique
solution of \eqref{BSDE}, see Notation \ref{N55bis}.

\begin{theorem}\label{Defuv}
Let $(f,g)$ satisfy $H^{lip}(\zeta,\eta)$.
There exists $u\in\mathcal{D}(\mathfrak{a})$
such that for any $(s,x)\in[0,T]\times E$
\begin{equation*}
\left\{\begin{array}{rcl}
\forall t\in [s,T]:  Y^{s,x}_t &=& u(t,X_t)\quad \mathbbm{P}^{s,x}\text{a.s.}  \\
 M^{s,x}&=&M[u]^{s,x},
\end{array}\right.
\end{equation*}
and in particular $\frac{d\langle  M^{s,x},M[\psi]^{s,x}\rangle}{dV}=\mathfrak{G}^{\psi}(u)(\cdot,X_{\cdot})$  $dV\otimes d\mathbbm{P}^{s,x}$  a.e. on $[s,T]$. Moreover, for every $(s,x)$, $Y^{s,x}_s$ is $\mathbbm{P}^{s,x}$ a.s. equal to a constant (which we shall still denote $Y^{s,x}_s$) 
and $u(s,x) =  Y^{s,x}_s$ for every $(s,x) \in [0,T] \times E$.
\end{theorem}

\begin{corollary}\label{uvBSDE}
Let $(f,g)$ satisfy $H^{lip}(\zeta,\eta)$.
 For any $(s,x)\in[0,T]\times E$, 
 the  functions $u$ obtained in Theorem \ref{Defuv} satisfies $\mathbbm{P}^{s,x}$ a.s.  on $[s,T]$
\begin{equation*}
u(t,X_t) = g(X_T) + \int_t^T f\left(r,X_r,u(r,X_r),\mathfrak{G}^{\psi}(u)(r,X_r)\right)dV_r  -(M[u]^{s,x}_T - M[u]^{s,x}_t),
\end{equation*}
and in particular, $\mathfrak{a}(u)=-f(\cdot,\cdot,u,\mathfrak{G}^{\psi}(u))$.
\end{corollary}
\begin{proof}
The corollary follows from Theorem \ref{Defuv} and Lemma \ref{ModifImpliesdV}. 
\end{proof}

\begin{theorem}\label{MainTheorem}
Let $(\mathbbm{P}^{s,x})_{(s,x)\in[0,T]\times E}$ be a canonical Markov class associated to a transition kernel measurable 
in time (see Definitions 3.4, 3.5 and 3.7 in \cite{paperAF}) which
solves a martingale problem associated with
the triplet $(\mathcal{D}(a),a,V)$.
Moreover we suppose  Hypothesis $H^{mom}(\zeta,\eta)$ for some positive
 $\zeta,\eta$.  Let $(f,g)$ be a couple satisfying $H^{lip}(\zeta,\eta)$.
\\
\\
Then $Pseudo-PDE(f,g)$ has a unique decoupled mild solution given by 
\begin{equation} \label{E525}
u:
\begin{array}{ccl}
[0,T]\times E&\longrightarrow& \mathbbm{R}\\
(s,x)&\longmapsto& Y^{s,x}_s,
\end{array}
\end{equation}
where $(Y^{s,x}, M^{s,x})$ denotes the (unique) solution of $BSDE^{s,x}(f,g)$ for fixed $(s,x)$.
\end{theorem}
\begin{proof}
Let $u$ be the function exhibited in Theorem \ref{Defuv}.
In order to show that $u$ is a decoupled mild solution of $Pseudo-PDE(f,g)$, it is enough by Proposition \ref{MartingaleImpliesMild} to show that it is a martingale solution.

In Corollary \ref{uvBSDE}, we have already seen that $\mathfrak{a}(u)= -f(\cdot,\cdot,u,\mathfrak{G}^{\psi}(u))$.
Concerning the second line of \eqref{PDEextended},  for any $x\in E$, 
we have
\\
$u(T,x)=u(T,X_T)=g(X_T)=g(x)$ $\mathbbm{P}^{T,x}$ a.s., so $u(T,\cdot)=g$, in the deterministic pointwise sense. 

We now show uniqueness. By Proposition \ref{MildImpliesMartingale}, it is enough to show that $Pseudo-PDE(f,g)$ admits at most one martingale solution.
Let $u,u'$ be two martingale solutions of $Pseudo-PDE(f,g)$. We fix $(s,x)\in[0,T]\times E$. By Proposition \ref{MartImpliesBSDE}, both couples, indexed by $[s,T]$, 
\\
$\left(u(\cdot,X),\quad u(\cdot,X)-u(s,x)+\int_s^{\cdot}f(\cdot,\cdot,u,\mathfrak{G}^{\psi}(u))(r,X_r)dV_r\right)$ and  
\\
$\left(u'(\cdot,X),\quad u'(\cdot,X)-u'(s,x)+\int_s^{\cdot}f(\cdot,\cdot,u',\mathfrak{G}^{\psi}(u))(r,X_r)dV_r\right)$ 
admit 
a $\mathbbm{P}^{s,x}$-version  which solves 
$BSDE^{s,x}(f,g)$ on $[s,T]$. By  Theorem \ref{uniquenessBSDE} and Remark \ref{BSDESmallInt},  $BSDE^{s,x}(f,g)$ admits a unique solution, so  $u(\cdot,X_{\cdot})$ and $u'(\cdot,X_{\cdot})$ are 
$\mathbbm{P}^{s,x}$-modifications one of the other on $[s,T]$. In particular, considering their values at time $s$, we have $u(s,x)=u'(s,x)$. We therefore have $u'=u$.
\end{proof}

\begin{corollary}\label{ClassicUnique} 
	Let $(f,g)$ satisfy $H^{lip}(\zeta,\eta)$.
	There is at most one classical solution of $Pseudo-PDE(f,g)$ and this only possible classical solution is the unique decoupled mild solution $(s,x)\longmapsto Y^{s,x}_s$, where $(Y^{s,x}, M^{s,x})$ denotes the (unique) solution of $BSDE^{s,x}(f,g)$ for fixed $(s,x)$.
\end{corollary}
\begin{proof}
	The proof follows from Proposition \ref{CoroClassic} and Theorem \ref{MainTheorem}.
\end{proof}

\begin{remark} \label{R524}
Let $(u,v)$ be the unique solution of the identification problem $IP(f,g)$, then $v$ also admits
a stochastic representation. Indeed, for every $(s,x)\in[0,T]\times E$, on  $[s,T]$,
\\
$\frac{d\langle M^{s,x},M^{s,x}[\psi]\rangle}{dV}=v(\cdot,X_{\cdot})$ $dV\otimes d\mathbbm{P}^{s,x}$ a.e. where $M^{s,x}$ is the second item of the solution of $BSDE^{s,x}(f,g)$.
\end{remark}

The existence of a decoupled mild solution of $Pseudo-PDE(f,g)$ provides in fact
an existence theorem for $BSDE^{s,x}(f,g)$ for any $(s,x)$.
The following constitutes  the converse of Theorem \ref{MainTheorem}.

\begin{proposition}\label{MildImpliesBSDE}
Assume $(f,g)$  satisfies $H^{mom}(\zeta,\eta)$.
Let $(u,v)$ be a solution of $IP(f,g)$, let $(s,x)\in[0,T]\times E$ and the associated probability $\mathbbm{P}^{s,x}$ be fixed. 
The couple
\begin{equation}
\left(u(t,X_t),\quad u(t,X_t)-u(s,x)+\int_s^{t}f(\cdot,\cdot,u,v)(r,X_r)dV_r\right)_{t\in[s,T]}
\end{equation}
has a $\mathbbm{P}^{s,x}$-version which solves $BSDE^{s,x}(f,g)$ on $[s,T]$.
\\
\\
In particular if $(f,g)$ satisfies the stronger hypothesis $H^{lip}(\zeta,\eta)$ and $(u,v)$ is the unique solution of $IP(f,g)$, then for any $(s,x)\in[0,T]\times E$,
\\
 $\left(u(t,X_t),\quad u(t,X_t)-u(s,x)+\int_s^{t}f(\cdot,\cdot,u,v)(r,X_r)dV_r\right)_{t\in[s,T]}$ is  a $\mathbbm{P}^{s,x}$-modification of the unique solution of $BSDE^{s,x}(f,g)$ on $[s,T]$.
\end{proposition}

\begin{proof}
It follows from Propositions  \ref{MildImpliesMartingale}, and \ref{MartImpliesBSDE}.

\end{proof}

\section{Examples of applications}\label{S5}

We now develop some examples.
In all the items  below there will be a canonical Markov class with 
 transition kernel being measurable in time which is solution of a  martingale Problem associated to some triplet
$(\mathcal{D}(a),a,V)$ as introduced in Definition \ref{MartingaleProblem}. 
Therefore all the results of
this paper will apply to all the examples below. In particular, Propositions  \ref{MildImpliesMartingale}, \ref{CoroClassic}, Theorem \ref{MainTheorem},  Corollary \ref{ClassicUnique} and Proposition \ref{MildImpliesBSDE} will apply but we will mainly emphasize Theorem \ref{MainTheorem} and Corollary \ref{ClassicUnique}.
In all the examples   $T > 0$ will be fixed.

\subsection{A new approach to Brownian BSDEs and associate semilinear PDEs}\label{S5a}

In this first application, the state space will be $E:=\mathbbm{R}^d$ for some $d\in\mathbbm{N}^*$.
\begin{notation}
A function $\phi\in\mathcal{B}([0,T]\times\mathbbm{R}^d,\mathbbm{R})$ will be said to have \textbf{polynomial growth} if there exists $p\in\mathbbm{N}$ and $C>0$ such that for every $(t,x)\in[0,T]\times\mathbbm{R}^d$,
 $|\phi(t,x)|\leq C( 1+\|x\|^p)$.
For any $k,p\in\mathbbm{N}$, $\mathcal{C}^{k,p}([0,T]\times\mathbbm{R}^d)$ (resp. $\mathcal{C}_b^{k,p}([0,T]\times\mathbbm{R}^d)$, resp. $\mathcal{C}_{pol}^{k,p}([0,T]\times\mathbbm{R}^d)$) will denote the sublinear algebra of $\mathcal{C}([0,T]\times\mathbbm{R}^d,\mathbbm{R})$ of functions admitting continuous (resp. bounded continuous, resp. continuous with polynomial growth) derivatives up to order $k$ in the first variable and order $p$ in the second.
\end{notation}

We consider bounded Borel functions $\mu\in\mathcal{B}_b([0,T]\times \mathbbm{R}^d, \mathbbm{R}^d)$ and $\alpha\in\mathcal{B}_b([0,T]\times \mathbbm{R}^d, S^+_d(\mathbbm{R}))$ where $ S^+_d(\mathbbm{R})$ is the space of symmetric non-negative  $d \times d$ real matrices.
We define for  $\phi \in \mathcal{C}^{1,2}([0,T]\times\mathbbm{R}^d)$ the operator $a$
 by
\begin{equation} \label{PIDE}
a(\phi)=\partial_t\phi + \frac{1}{2}\underset{i,j\leq d}{\sum} \alpha_{i,j}\partial^2_{x_ix_j}\phi + \underset{i\leq d}{\sum} \mu_i\partial_{x_i}\phi.
\end{equation}
We will assume the following.
\begin{hypothesis}\label{HypDiff}
There exists a canonical Markov class $(\mathbbm{P}^{s,x})_{(s,x)\in[0,T]\times\mathbbm{R}^d}$ which solves the Martingale Problem associated to $(\mathcal{C}_b^{1,2}([0,T]\times\mathbbm{R}^d),a,V_t \equiv t)$ in the sense of Definition \ref{MartingaleProblem}.
\end{hypothesis}

We now recall a non-exhaustive list of sets of conditions on $\mu,\alpha$ under which Hypothesis \ref{HypDiff} is satisfied.
\begin{enumerate}
\item  $\alpha$ is continuous non-degenerate, in the sense that for any $t,x$, $\alpha(t,x)$ is invertible, see Theorem 4.2 in \cite{stroock1975diffusion}; 
\item $\mu$  and $\alpha$ are continuous in the second variable, see Exercise 12.4.1  in \cite{stroock};
\item $d=1$ and $\alpha$ is uniformly positive on compact sets, see Exercise 7.3.3 in \cite{stroock}.
\end{enumerate}

\begin{remark}  \label{R63a}
\begin{itemize}
\item 
When the item 1. or 3. above 
is satisfied, the  mentioned canonical Markov class is unique, but whenever
only 2. holds, uniqueness may fail.
\item We emphasize that given a fixed canonical Markov class, we obtain well-posedness results 
concerning the martingale solution (and so the 
decoupled mild solution) of an associated PDE.
\item Nevertheless, for every  canonical Markov class solving the martingale problem could correspond a different solution.
\end{itemize}
\end{remark}

In this context,  for $\phi,\psi$ in $\mathcal{D}(a)$, the carr\'e du champs operator (see Definition \ref{CarreDuChamp}) is given by $\Gamma(\phi,\psi) = \underset{i,j\leq d}{\sum}\alpha_{i,j}\partial_{x_i}\phi\partial_{x_j}\psi$.

\begin{remark}\label{BigDomainDiff}
By a localization procedure, it is also clear that for every \\ $(s,x)\in[0,T]\times\mathbbm{R}^d$, 
for any $\phi\in \mathcal{C}^{1,2}([0,T]\times\mathbbm{R}^d)$, $\phi(\cdot,X_{\cdot})-\int_s^{\cdot}a(\phi)(r,X_r)dr\in\mathcal{H}^2_{loc}$ with respect to $\mathbbm{P}^{s,x}$.
Consequently Proposition \ref{bracketindomain} 
extends to all $\phi\in \mathcal{C}^{1,2}([0,T]\times\mathbbm{R}^d)$.
\end{remark}
We set now $\mathcal{D}(a)=\mathcal{C}_{pol}^{1,2}([0,T]\times\mathbbm{R}^d)$.

For any $i\in[\![ 1;d]\!]$,  the function $Id_i$ denotes $(t,x)\longmapsto x_i$ which belongs to $\mathcal{D}(a)$ and $Id:=(Id_1,\cdots,Id_d)$.
It is clear that for any $i$, $a(Id_i)=\mu_i$, and for any $i,j$, $Id_iId_j\in\mathcal{D}(a)$ and $\Gamma(Id_i,Id_j)=\alpha_{i,j}$. In particular, by Corollary \ref{H2Vloc}, $(Id_1,\cdots,Id_d)$ satisfy Hypothesis \ref{HypBrackPhi} and,
since $\mu,\alpha$ are bounded, they satisfy Hypothesis \ref{HypPhi}.

For any $i$ we can therefore consider the MAF 
$M[Id_i]:(t,u)\mapsto X^i_u-X^i_t-\int_t^u\mu_i(r,X_r)dr$ whose c\`adl\`ag version under $\mathbbm{P}^{s,x}$ 
for every $(s,x)\in[0,T]\times\mathbbm{R}^d$  is 
 $M[Id_i]^{s,x}=\mathds{1}_{[s,T]}\left(X^i-x_i-\int_s^{\cdot}\mu_i(r,X_r)dr\right)$ and for any
 $i,j$ we have $\langle M[Id_i]^{s,x},M[Id_j]^{s,x}\rangle=\int_s^{\cdot\vee s}\alpha_{i,j}(r,X_r)dr$.

\begin{lemma}\label{momsde}
	Let $(s,x)\in[0,T]\times\mathbbm{R}^d$ and associated probability  $\mathbbm{P}^{s,x}$, $i\in[\![1;d]\!]$ and $p\in[1,+\infty[$ be fixed. Then $\underset{t\in[s,T]}{\text{sup }}|X^i_t|^p\in L^1$.
\end{lemma}
\begin{proof}
	We have $X^i=x_i+\int_s^{\cdot}\mu_i(r,X_r)dr+M[Id_i]^{s,x}$ where $\mu_i$ is bounded so it is enough to show that $\underset{t\in[s,T]}{\text{sup }}|M[Id_i]^{s,x}_t|^p\in L^1$. Since 
$\langle M[Id_i]^{s,x}\rangle = \int_s^{\cdot\vee s}\alpha_{i,i}(r,X_r)dr$,
 which is bounded, the result holds by Burkholder-Davis-Gundy inequality.
\end{proof}

\begin{corollary}
$(\mathbbm{P}^{s,x})_{(s,x)\in[0,T]\times\mathbbm{R}^d}$ solves the Martingale Problem associated to $(\mathcal{C}_{pol}^{1,2}([0,T]\times\mathbbm{R}^d),a,V_t \equiv t)$ in the sense of Definition \ref{MartingaleProblem}.
\end{corollary}
\begin{proof}
By Remark \ref{BigDomainDiff}, for any $\phi\in\mathcal{C}_{pol}^{1,2}([0,T]\times\mathbbm{R}^d)$ and $(s,x)\in[0,T]\times\mathbbm{R}^d$, $\phi(\cdot,X_{\cdot})-\int_s^{\cdot}a(\phi)(r,X_r)dr$  is a 
 $\mathbbm{P}^{s,x}$-local martingale. 
Since $\phi$ and $a(\phi)$ have polynomial growth, Lemma \ref{momsde} and Jensen's inequality imply that it is also a square integrable martingale.
\end{proof}
We now consider a couple $(f,g)$ satisfying
$H^{lip}(\|\cdot\|^p,\|\cdot\|^p)$ 
for some $p \ge 1$. 
In this case   Hypothesis \ref{Hpq} can be retranslated into what follows.

\begin{itemize} 
	\item $g$ is Borel with polynomial growth;
	\item $f$ is Borel with polynomial growth in $x$ (uniformly in $t$), and Lipschitz in $y,z$.
\end{itemize}
We consider the PDE
\begin{equation}\label{PDEparabolic}
\left\{\begin{array}{lcr}
\partial_tu + \frac{1}{2}\underset{i,j\leq d}{\sum} \alpha_{i,j}\partial^2_{x_ix_j}u + \underset{i\leq d}{\sum} \mu_i\partial_{x_i}u +f(\cdot,\cdot,u,\alpha\nabla u)=0\\
u(T,\cdot)=g.
\end{array}\right.
\end{equation}
We emphasize that for $u\in\mathcal{C}_{pol}^{1,2}([0,T]\times\mathbbm{R}^d)$, $\alpha\nabla u=\Gamma^{Id}(u)$. The associated decoupled mild equation is given by

\begin{equation}
\left\{
\begin{array}{rcl}
u(s,x)&=&P_{s,T}[g](x)+\int_s^TP_{s,r}\left[f\left(\cdot,\cdot,u,v\right)(r,\cdot)\right](x)dr\\
u(s,x)x_i &=&P_{s,T}[gId_i](x) -\int_s^TP_{s,r}\left[\left(v_i+u\mu_i-Id_if\left(\cdot,\cdot,u,v\right)\right)(r,\cdot)\right](x)dr, i\in[\![1;d]\!],
\end{array}\right.
\end{equation}
$(s,x)\in[0,T]\times\mathbbm{R}^d$,
where $P$ is the transition kernel of the canonical Markov class.

\begin{proposition}\label{PropBrownianExample}
Assume the validity of Hypothesis \ref{HypDiff} and that  $(f,g)$ satisfies 
\\
$H^{lip}(\|\cdot\|^p,\|\cdot\|^p)$
 for some $p \ge 1$.
 Then equation \eqref{PDEparabolic} has a unique decoupled mild solution $u$. 

Moreover it has at most one classical solution which (when it exists)
 equals this function $u$.
\end{proposition}

\begin{proof}
	$(\mathbbm{P}^{s,x})_{(s,x)\in[0,T]\times\mathbbm{R}^d}$ solves a 
	martingale problem in the sense of Definition \ref{MartingaleProblem} and has
	a transition kernel which is measurable in time.
	Moreover $(Id_1,\cdots,Id_d)$ fulfills Hypothesis \ref{HypPhi}, 
	$(\mathbbm{P}^{s,x})_{(s,x)\in[0,T]\times\mathbbm{R}^d}$ satisfies (by Lemma \ref{momsde}) 
	$H^{mom}(\|\cdot\|^p,\|\cdot\|^p)$
	for some $p \ge 1$ 
	and $(f,g)$ satisfies $H^{lip}(\|\cdot\|^p,\|\cdot\|^p)$.
	So Theorem  \ref{MainTheorem} and Corollary \ref{ClassicUnique} apply.
\end{proof}

\begin{remark}
The unique decoupled mild solution mentioned in the previous proposition admits the probabilistic representation given in Theorem \ref{MainTheorem}.
\end{remark}

\begin{remark} \label{R58}
	In the classical literature, 
the Brownian BSDE \eqref{BSDEIntroN}
has been related to a slightly different type of  parabolic PDE, i.e.
	\begin{equation}\label{PDEexamples}
	\left\{
	\begin{array}{l}
	\partial_tu + \frac{1}{2}\underset{i,j\leq d}{\sum} (\sigma\sigma^\intercal)_{i,j}\partial^2_{x_ix_j}u + \underset{i\leq d}{\sum} \mu_i\partial_{x_i}u + f(\cdot,\cdot,u,\sigma\nabla u)=0 \\
	u(T,\cdot) = g, 
	\end{array}\right.
	\end{equation}
 (where $\sigma=\sqrt{\alpha}$ in the sense of non-negative symmetric matrices)	rather than \eqref{PDEparabolic}. In fact, the only difference is that
 the term $\sigma\nabla u$ replaces $\alpha\nabla u$ in the fourth
argument of the driver 
 $f$.  See the introduction of the present paper, or \cite{pardoux1992backward} for more details.

Our methodology also allows to represent \eqref{PDEexamples}.
Under the probability $\mathbbm{P}^{s,x}$ (for some fixed $(s,x)$), one can introduce the square integrable martingale $\tilde{M}[Id]^{s,x}:=\int_s^{\cdot}(\sigma^\intercal)^+(r,X_r)dM[Id]^{s,x}_r$ where $A\mapsto A^+$ denotes the Moore-Penrose pseudo-inverse operator, see \cite{PseudoInverse} chapter 1.
The Brownian BSDE \eqref{BSDEIntroN} can then be reexpressed here as
\begin{equation} \label{BSDEIntroNN}
Y^{s,x}_t=g(X_T)+\int_t^Tf\left(r,X_r,Y^{s,x}_r,\frac{d\langle M^{s,x},\tilde{M}[Id]^{s,x}\rangle_r}{dr}\right)dr-(M^{s,x}_T-M^{s,x}_t).
\end{equation}
Under the assumptions of Proposition \ref{PropBrownianExample} 
where $\alpha=\sigma\sigma^\intercal$,
it is possible to show that \eqref{BSDEIntroNN} constitutes the
probabilistic representation of \eqref{PDEexamples} 
performing similar arguments as in our approach for
\eqref{PDEparabolic}. In particular we can show existence
and uniqueness of a function $u\in\mathcal{L}^2_X$ for which there exists $v_1,\cdots,v_d\in\mathcal{L}^2_X$ such that for all $(s,x)\in[0,T]\times \mathbbm{R}^d$,
\begin{equation} \label{MildBrownian}
\left\{
\begin{array}{rcl}
u(s,x)&=&P_{s,T}[g](x)+\int_s^TP_{s,r}\left[f\left(\cdot,\cdot,u,(\sigma^\intercal)^+v\right)(r,\cdot)\right](x)dr\\
u(s,x)x_i &=&P_{s,T}[gId_i](x) -\int_s^TP_{s,r}\left[\left(v_i+u\mu_i-Id_if\left(\cdot,\cdot,u,(\sigma^\intercal)^+v\right)\right)(r,\cdot)\right](x)dr, i\in[\![1;d]\!].
\end{array}\right.
\end{equation}
\eqref{MildBrownian} constitutes indeed a suitable decoupled
mild formulation corresponding to \eqref{PDEexamples}.
Moreover, this function $u$, whenever it belongs to $\mathcal{C}_{pol}^{1,2}([0,T]\times\mathbbm{R}^d)$,
 is the unique classical solution of \eqref{PDEexamples}.

This technique is however technically more complicated and
for purpose of illustration we prefer to remain  in our setup (which is by the way close to \eqref{PDEexamples})
to keep our notion of decoupled-mild solution more comprehensible.

\end{remark}

\begin{remark}
It is also possible to treat jump diffusions instead of continuous 
diffusions (see \cite{stroock1975diffusion}), and under suitable conditions
 on the coefficients, it is also possible to prove existence and uniqueness of a decoupled mild solution for equations of type
	\begin{equation}
		\left\{\begin{array}{l}
			\partial_tu + \frac{1}{2}Tr(\alpha\nabla^2u) + (\mu,\nabla u) +\int\left(u(\cdot,\cdot+y)-u-\frac{(y,\nabla u)}{1+\|y\|^2}\right)K(\cdot,\cdot,dy) + f(\cdot,\cdot,u,\Gamma^{Id}(u))=0\\
			u(T,\cdot)=g,
		\end{array}\right.
	\end{equation}
where $K$ is a L\'evy kernel: this means  that for every $(t,x)\in [0,T]\times \mathbbm{R}^d$, $K(t,x,\cdot)$ is a $\sigma$-finite measure 
on $\mathbbm{R}^d\backslash\{0\}$, $\underset{t,x}{\text{sup}}\int \frac{\|y\|^2}{1+\|y\|^2}K(t,x,dy)<\infty$ and for every Borel set $A\in\mathcal{B}(\mathbbm{R}^d\backslash\{0\})$, 
$(t,x)\longmapsto \int_A \frac{\|y\|^2}{1+\|y\|^2}K(t,x,dy)$ is Borel.
In that framework we have 
\begin{equation}
	\Gamma^{Id}:\phi\longmapsto\alpha\nabla \phi+\left(\int y_i(\phi(\cdot,\cdot+y)-\phi(\cdot,\cdot))K(\cdot,\cdot,dy)\right)_{i\in[\![ 1;d]\!]}.
\end{equation}

\end{remark}

\subsection{Parabolic semi-linear PDEs with distributional drift}\label{S5b} 
 
The context of this subsection is the one introduced by
 Flandoli, Russo \& Wolf in \cite{frw1} and \cite{frw2}), see also  \cite{russo_trutnau07}, \cite{diel} for  recent developments. 
We refer to Section 4.3
of \cite{paper2} for 
  a more detailed introduction.
In particular \cite{frw1, frw2} consider stochastic differential equations
with distributional drift, whose solution are possibly
non-semimartingales.
 These authors introduced a suitable framework of 
a martingale problem related to a PDE operator involving 
 a distributional drift $b'$ which is the
derivative of a continuous function. 
\cite{issoglio} approached the $n$-dimensional setting for the first time 
and later developments were discussed by \cite{cannizzaro2015}
studying singular SDEs involving paracontrolled distributions. 
Other Markov processes associated to diffusion operators
which are not semimartingales were produced
when the diffusion operator is in divergence form, see
 e.g.  \cite{rozkosz}.
 \cite{wurzer}  linked  second order ODEs
with a distributional coefficient  and BSDEs.
		In those BSDEs the final horizon was a stopping time. 
\cite{issoglio_jing16} and 
\cite{issoglio}  have considered a class of BSDEs
		involving distributions in their setting.

Let $b,\sigma\in \mathcal{C}^0(\mathbbm{R})$ such that $\sigma>0$.
In \cite{frw1}, the authors introduce a (generalized) notion
for the equation $L f = \ell$, for $f \in C^1({\mathbb R})$.
They suppose the existence of a function
$\Sigma:{\mathbb R} \rightarrow {\mathbb R}$ which formally equals
$2 \int_0^\cdot  \frac{b'}{\sigma^2}(y) dy$ and it is defined
via mollification. A typical situation when $\Sigma$ exists  
arises when either $b$ or $\sigma^2$ have locally bounded variation.
If $\Sigma$ exists then the function $h: {\mathbb R} \rightarrow {\mathbb R}$
defined by $h(0) = 0$ and $h' = e^{-\Sigma}$ is
 $L$-{\it harmonic} function,
in the sense that it fulfills  $Lh = 0$, see Proposition 2.3 of \cite{frw1}.
$\mathcal{D}_L$ is defined as the set of $f\in\mathcal{C}^1(\mathbbm{R})$ such that there exists some $\ell \in\mathcal{C}^0(\mathbbm{R})$ with $Lf= \ell$
 and it is a linear algebra.

 Let $v$ be the unique solution to $Lv=1\text{,  }v(0)=v'(0)=0$, see Remark 2.4 in \cite{frw1}; we will assume
\begin{equation}\label{NonExplosion}
v(-\infty)=v(+\infty)=+\infty,
\end{equation}
which represents a non-explosion condition.
In this case, Proposition 3.13 in \cite{frw1} states that a certain martingale
problem associated to 
$(\mathcal{D}_L,L)$ is well-posed. Its solution will be denoted $(\mathbbm{P}^{s,x})_{(s,x)\in[0,T]\times\mathbbm{R}^d}$.

The canonical process $X$ is a $\mathbbm{P}^{s,x}$-Dirichlet process for every $(s,x)$,
i.e. the sum of a local martingale and a zero quadratic variation
process and it is a semimartingale if and only if $\Sigma$ is locally
of bounded variation, see Corollary 5.11 in \cite{frw2}.
$(\mathbbm{P}^{s,x})_{(s,x)\in[0,T]\times\mathbbm{R}^d}$  defines a canonical Markov class and  Proposition B.2 in \cite{paper2} implies that its transition kernel is measurable in time.

We introduce below the domain that we will indeed use.
\begin{definition}\label{domain}
	We set 
	\begin{equation}
	\mathcal{D}^{max}(a)=\left\{
	\phi\in\mathcal{C}^{1,1}([0,T]\times\mathbbm{R}):\frac{\partial_x\phi}{h'}\in\mathcal{C}^{1,1}([0,T]\times\mathbbm{R})\right\}.
	\end{equation}
On $\mathcal{D}^{max}(a)$, we set 
$L\phi:=\frac{\sigma^2h'}{2}\partial_x(\frac{\partial_x\phi}{h'})$ and $a(\phi):=\partial_t\phi+L\phi$. We then define the smaller domain
	\begin{equation}
	\mathcal{D}(a)=\left\{
	\phi\in\mathcal{D}^{max}(a):\sigma\partial_x\phi\in\mathcal{C}^{0,0}_{pol}([0,T]\times\mathbbm{R})\right\}.
	\end{equation}
	\end{definition}
We formulate here some supplementary assumptions that we will make, the first one being called (TA) in \cite{frw1}.
\begin{hypothesis}\label{TA}
\leavevmode
\begin{itemize}
\item There exists $c_1, C_1>0$ such that $c_1 \le \sigma h' \le C_1$;
\item $\sigma$ has linear growth.
\end{itemize}
\end{hypothesis}
The first item states in particular that $\sigma h'$ is bounded so $h\in\mathcal{D}(a)$. Proposition 3.2 in \cite{frw1} states that for every $(s,x)$, $\langle M[h]^{s,x}\rangle=  \int_s^{\cdot\vee s}(\sigma h')^2(X_r)dr$.
Moreover the AF defined by $\langle M[h] \rangle_u^t=  \int_t^u(\sigma h')^2(X_r)dr,$
is absolutely continuous with respect to $\hat V_t \equiv t$.
Therefore Hypothesis \ref{HypBrackPhi} is satisfied (for $\psi = h$) and $\mathfrak{G}^h(h)=(\sigma h')^2$. Since this function is 
bounded and clearly $a(h)=0$ then $h$ satisfies Hypothesis \ref{HypPhi}.

We will therefore consider the $h$-generalized gradient $\Gamma^h$ associated to $a$; Proposition 4.23 in \cite{paper2} implies the following.
\begin{proposition} \label{PDistrDrift}
	Let $\phi\in\mathcal{D}(\Gamma^h)$, then $\Gamma^h(\phi)=\sigma^2h'\partial_x\phi$.
\end{proposition}

The deterministic equation considered in this section is a  semilinear PDE with singular (or distributional) drift $b'$ given by
\begin{equation}\label{PDEdistri}
\left\{
\begin{array}{l}
\partial_tu + \frac{1}{2}\sigma^2\partial^2_x u + b'\partial_xu +f(\cdot,\cdot,u,\sigma^2 h'\partial_xu)=0\quad\text{ on }[0,T]\times\mathbbm{R}\\
u(T,\cdot) = g.
\end{array}\right.
\end{equation}

The associated PDE in the decoupled mild sense is given by
\\
\begin{equation}
\left\{
\begin{array}{rcl}
u(s,x)&=&P_{T-s}[g](x)+\int_s^TP_{r-s}\left[f\left(\cdot,\cdot,u,v\right)(r,\cdot)\right](x)dr\\
u(s,x)h(x) &=&P_{T-s}[gh](x) -\int_s^TP_{r-s}\left[\left(v-hf\left(\cdot,\cdot,u,v\right)\right)(r,\cdot)\right](x)dr,
\end{array}\right.
\end{equation}
$(s,x)\in[0,T]\times\mathbbm{R}$,
where $P$ is the (time-homogeneous) transition kernel of the  canonical Markov class.
\\
\\
In order to consider the $BSDE^{s,x}(f,g)$ for functions $(f,g)$  having polynomial growth in $x$, we had shown in \cite{paper2} the following result, stated
 as Proposition 4.24.

\begin{proposition}\label{MomentsDistri}
	We suppose that Hypothesis \ref{TA} is fulfilled. 
	Then, for any  $p\in\mathbbm{N}$ and   $(s,x)\in[0,T]\times\mathbbm{R}$, $\mathbbm{E}^{s,x}[|X_T|^p]<\infty$ and 
	$\mathbbm{E}^{s,x}[\int_s^T|X_r|^pdr]<\infty$. In other words, for any $p\geq 1$, the canonical Markov class $(\mathbbm{P}^{s,x})_{(s,x)\in[0,T]\times\mathbbm{R}^d}$ satisfies $H^{mom}(|\cdot|^p,|\cdot|^p)$, see Hypothesis \ref{HypMom}.
\end{proposition}

Next we have the following.
\begin{proposition}\label{MPnewdomaindistri}
We suppose that Hypothesis \ref{TA} is fulfilled. Then 
	$(\mathbbm{P}^{s,x})_{(s,x)\in[0,T]\times\mathbbm{R}^d}$  solves the Martingale Problem associated to $(a,\mathcal{D}(a),V_t\equiv t)$ in the sense of Definition \ref{MartingaleProblem}.
\end{proposition}
\begin{proof}
Let $(s,x)\in[0,T]\times \mathbbm{R}$ be fixed.
Proposition 4.23 in \cite{paper2} implies that for any $\phi\in\mathcal{D}(a)$, $\phi(\cdot,X_{\cdot})-\int_s^{\cdot}a(\phi)(r,X_r)dr$ is a (continuous) $\mathbbm{P}^{s,x}$-local martingale, so taking Definition \ref{MartingaleProblem} into account, it is enough to show that  this local martingale is a square integrable martingale.
Considering Definition 4.21, Proposition 4.22 and Proposition 2.6 in \cite{paper2}, we know that the angular bracket of this local martingale is given by $\int_s^{\cdot}(\sigma\partial_x\phi)^2(r,X_r)dr$. Since $\phi\in\mathcal{D}(a)$ then $\sigma\partial_x\phi$ has polynomial growth, so by Proposition \ref{MomentsDistri}, $\int_s^T(\sigma\partial_x\phi)^2(r,X_r)dr\in L^1$ and this implies that the 
aforementioned local martingale is a square integrable martingale.
\end{proof}

We can now state the  main result of this section. 
\begin{proposition} \label{P616}
	Assume the non-explosion condition \eqref{NonExplosion}, Hypothesis \ref{TA} and that $(f,g)$ satisfies
 $H^{lip}(|\cdot|^p,|\cdot|^p)$  for some  $p\geq 1$,
see Hypothesis \ref{Hpq}. Then, equation \eqref{PDEdistri} has a unique decoupled mild solution $u$.
 Moreover, there is at most one classical solution which can only be equal to $u$.
\end{proposition}

\begin{proof}
	The assertions come from Theorem \ref{MainTheorem} and Corollary
 \ref{ClassicUnique} which applies thanks to Propositions \ref{MPnewdomaindistri}, \ref{MomentsDistri}, and the fact that $h$ satisfies Hypothesis \ref{HypPhi}.
\end{proof}
\begin{remark} \label{R616}
The unique decoupled mild solution $u$ can be of course  represented
by \eqref{E525},  Theorem \ref{MainTheorem}.

\end{remark}

\begin{appendices}

\section{Proof of Theorem
 \ref{uniquenessBSDE} and related technicalities}	
\label{A}

 We adopt here the same notations as at the beginning of Section \ref{S1}.
 We will denote 
 $L^2(d\hat V\otimes d\mathbbm{P}) $
  the quotient space of ${\mathcal L}^2(d\hat V\otimes d\mathbbm{P})$ with respect to the subspace of
  processes equal to zero $d\hat V\otimes d\mathbbm{P}$ a.e. \\
 $L^2(d\hat V\otimes d\mathbbm{P})$ is a Hilbert space equipped
 with its usual norm. 
  $L^{2,cadlag}(d\hat V\otimes d\mathbbm{P})$) will stand for the subspace of  $L^{2}(d\hat V\otimes d\mathbbm{P})$)  of elements having a c\`adl\`ag representative. We emphasize that  $L^{2,cadlag}(d\hat V\otimes d\mathbbm{P})$ is not a closed subspace of $L^{2}(d\hat V\otimes d\mathbbm{P})$.
The application which to a process associate its class will be denoted $\phi\mapsto\dot{\phi}$.

\begin{proposition}\label{BSDEexpectations}
If $(Y,M)$ solves $BSDE(\xi,\hat{f}, V,\hat{M})$, and if we denote
\\
 $\hat{f}\left(r,\cdot,Y_r,\frac{d\langle M,\hat{M} \rangle}{d\hat V}(r)\right)$ by $\hat{f}_r$, then for any $t\in[0,T]$, a.s. we have 
\begin{equation} \label{E32bis}
\left\{\begin{array}{rcl}
Y_t &=& \mathbbm{E}\left[\xi+\int_t^T\hat{f}_rd\hat V_r\middle|\mathcal{F}_t\right] \\
M_t &=& \mathbbm{E}\left[\xi+\int_0^T\hat{f}_rd\hat V_r\middle|\mathcal{F}_t\right]-\mathbbm{E}\left[\xi+\int_0^T\hat{f}_rd\hat V_r\middle|\mathcal{F}_0\right].
\end{array}\right.
\end{equation}
\end{proposition}
\begin{proof} 
	Since $Y_t=\xi +\int_t^T\hat{f}_rd\hat V_r - (M_T-M_t)$ a.s.,   $Y$ being
	an adapted process and $M$ a martingale, taking the expectation in \eqref{BSDEcadlag} at time $t$, we directly
	get  $Y_t = \mathbbm{E}\left[\xi+\int_t^T\hat{f}_rd\hat V_r\middle|\mathcal{F}_t\right]$ and in particular that $Y_0 = \mathbbm{E}\left[\xi+\int_0^T\hat{f}_rd\hat V_r\middle|\mathcal{F}_0\right]$. Since $M_0=0$, looking at the BSDE at time 0 we get
	$$ M_T 
	= \xi +\int_0^T\hat{f}_rd\hat V_r -\mathbbm{E}\left[\xi+\int_0^T\hat{f}_rd\hat V_r\middle|\mathcal{F}_0\right].$$
	Taking the expectation with respect to $\mathcal{F}_t$ in the above inequality,
	gives the second line of \eqref{E32bis}.
\end{proof}

\begin{lemma}\label{ZinL2}
	Let $M\in\mathcal{H}^2$ and $\phi$ be a bounded positive process.
	Then there exists a constant $C>0$ such that for any $i\in[\![ 1;d]\!]$, 
	\\
	$\int_0^T\phi_r\left(\frac{d\langle M,\hat{M}^i\rangle}{d\hat V}(r)\right)^2d\hat V_r\leq C\int_0^T\phi_rd\langle M\rangle_r$. In particular, $\frac{d\langle M,\hat{M}^i\rangle}{d\hat V}$ belongs to 
	$L^2(d\hat V\otimes d\mathbbm{P})$.
\end{lemma}
\begin{proof}
	We fix $i\in[\![ 1;d]\!]$.
	By Hypothesis \ref{HypBSDE} 
	$\frac{d\langle \hat{M}^i\rangle}{d\hat V}$ is bounded; using
	Proposition B.1  
	and Remark 3.3 in \cite{paper1preprint}, we show the existence of $C>0$ such that 
	\begin{equation}
	\begin{array}{rcl}
	\int_0^T\phi_r\left(\frac{d\langle M,\hat{M}^i\rangle}{d\hat V}(r)\right)^2d\hat V_r&\leq&\int_0^T\phi_r \frac{d\langle\hat{M}^i\rangle}{d\hat V}(r)\frac{d\langle M\rangle}{d\hat V}(r)d\hat V_r\\
	&\leq&C\int_0^T \phi_r\frac{d\langle M\rangle}{d\hat V}(r)d\hat V_r\\
	&\leq&C\int_0^T\phi_rd\langle M\rangle_r.
	\end{array}
	\end{equation}
	In particular, setting $\phi=1$, we have $\int_0^T\left(\frac{d\langle M,\hat{M}^i\rangle}{d\hat V}(r)\right)^2d\hat V_r\leq C\langle M\rangle_T$ which belongs to $L^1$ since $M\in\mathcal{H}^2_0$.
\end{proof}

We fix for now a couple $(\dot U,N)\in L^{2}(d\hat V\otimes d\mathbbm{P})\times \mathcal{H}^2_0$ and we consider a   representative $U$ of $\dot U$.
Until Proposition \ref{supY} included, we will use the notation
$\hat f_r := \hat{f}\left(r,\cdot,U_r,\frac{d\langle N,\hat{M} \rangle}{d\hat V}(r)\right)$.
\begin{proposition}\label{L1}
For any $t\in[0,T]$, $\int_t^T\hat{f}^2_r d\hat V_r$ 
belongs to
  $L^1$ and $\left(\int_t^T\hat{f}_r d\hat V_r\right)$ is in $L^2$.
\end{proposition}
\begin{proof}
	By Jensen's inequality and by Lemma \ref{ZinL2},
taking into account the Lipschitz conditions on $\hat{f}$ in 
	Hypothesis \ref{HypBSDE},
	there exist  positive constants
	$C,C',C"$ 
	such that, for any $t\in[0,T]$, we have  
	\\
	\begin{equation}
	\begin{array}{rcl}
 \left(\int_t^T\hat{f}_rd\hat V_r\right)^2
	 &\leq& C \int_t^T\hat{f}_r^2d\hat V_r
 \\
	&\leq& C'\left(\int_t^T\hat{f}^2\left(r,\cdot,0,0\right)d\hat V_r + \int_t^TU_r^2d\hat V_r + \underset{i=1}{\overset{d}{\sum}}\int_t^T\left(\frac{d\langle N,\hat{M}^i \rangle}{d\hat V}(r)\right)^2d\hat V_r\right)\\
	&\leq& C"\left(\int_t^T\hat{f}^2\left(r,\cdot,0,0\right)d\hat V_r + \int_t^TU_r^2d\hat V_r + (\langle N\rangle_T-\langle N\rangle_t)\right).
	\end{array}
	\end{equation}
	All terms on the right-hand side are in $L^1$.
	Indeed,  $N$ is taken in $\mathcal{H}^2$, $\dot{U}$  in $L^2(d\hat V\otimes d\mathbbm{P})$ and by Hypothesis \ref{HypBSDE}, $f(\cdot,\cdot,0,0)$ is in ${\mathcal L}^2(d\hat V\otimes d\mathbbm{P})$.
	This concludes the proof.
\end{proof}

We can therefore state
the following definition.

\begin{definition}\label{defYM} 
 The random function
\begin{equation} \label{Emartingale} 
t \mapsto  \mathbbm{E}\left[\xi + \int_0^T \hat{f}_r d\hat V_r\middle|\mathcal{F}_t\right]-\mathbbm{E}\left[\xi + \int_0^T \hat{f}_rd\hat V_r\middle|\mathcal{F}_0\right],
\end{equation}
 is a square integrable martingale by Proposition \ref{L1}. 
Since the stochastic basis fulfills
 the usual conditions, by  Theorem 4 in Chapter IV of \cite{dellmeyerB},
\eqref{Emartingale}   admits a c\`adl\`ag version, that we denote $M$.
We denote by $Y$ the c\`adl\`ag process defined by \\
$Y_t = \xi+\int_t^T\hat{f}_r d\hat V_r - (M_T - M_t)$.
This will be called the {\bf c\`adl\`ag reference process} and we will omit its 
dependence to $(\dot U,N)$.
\end{definition}

\begin{proposition}\label{L2}
$Y$ and $M$  are square integrable processes.
\end{proposition}

\begin{proof}
	We already know that $M$ is a square integrable martingale. 
	As we have seen in Proposition \ref{L1}, $\int_t^T\hat{f}_rd\hat V_r$ belongs to $L^2$ for any $t\in[0,T]$ and by Hypothesis \ref{HypBSDE}, $\xi\in L^2$. So by \eqref{E32bis} and   Jensen's inequality for conditional expectation we have
	\begin{equation*}
	\begin{array}{rcl}
	\mathbbm{E}\left[Y_t^2\right] &=& \mathbbm{E}\left[\mathbbm{E}\left[\xi+\int_t^T\hat{f}_rd\hat V_r\middle|\mathcal{F}_t\right]^2\right]\\
	& \leq & \mathbbm{E}\left[\mathbbm{E}\left[\left(\xi+\int_t^T\hat{f}_rd\hat V_r\right)^2\middle|\mathcal{F}_t\right]\right] \\
	& \leq &\mathbbm{E}\left[2 \xi^2+ 2 \int_t^T\hat{f}^2_rd\hat V_r \right],
	\end{array}
	\end{equation*}
	which is finite.
\end{proof}

\begin{proposition}\label{supY}
$\underset{t\in[0,T]}{\text{sup }}|Y_t|\in L^2$ and in particular, $Y\in\mathcal{L}^{2,cadlag}(d\hat{V}\otimes\mathbbm{P})$.
\end{proposition}
\begin{proof}
	
 Since 
	$dY_r=-\hat{f}_rd\hat V_r +dM_r$, by integration by parts formula we get
	\begin{equation*}
	d(Y_r^2e^{-\hat V_r})=-2e^{-\hat V_r}Y_r\hat{f}_rd\hat V_r +2e^{-\hat V_r}Y_{r^-}dM_r +e^{-\hat V_r}d[M]_r -e^{-\hat V_r}Y^2_rd\hat V_r.
	\end{equation*}
	So integrating from $0$ to some $t\in[0,T]$, yields
	\begin{equation*}
	\begin{array}{rcl}
	Y_t^2e^{-\hat V_t}&=& Y_0^2 - 2\int_0^t e^{-\hat V_r}Y_r\hat{f}_rd\hat V_r + 2\int_0^te^{-\hat V_r}Y_{r^-}dM_r\\
	&& +  \int_0^te^{-\hat V_r}d[ M]_r -  \int_0^t e^{-\hat V_r}Y^2_rd\hat V_r\\
	&\leq& Y_0^2 + \int_0^t e^{-\hat V_r}Y^2_rd\hat V_r + \int_0^t e^{-\hat V_r}\hat{f}^2_rd\hat V_r \\
	&&+ 2\left|\int_0^te^{-\hat V_r}Y_{r^-}dM_r\right| +  \int_0^te^{-\hat V_r}d[ M]_r -  \int_0^t e^{-\hat V_r}Y^2_rd\hat V_r\\
	&\leq& Z +   2\left|\int_0^te^{-\hat V_r}Y_{r^-}dM_r\right|,
	\end{array}        
	\end{equation*}
	where $Z=Y_0^2 + \int_0^T e^{-\hat V_r}\hat{f}^2_rd\hat V_r  +  \int_0^Te^{-\hat V_r}d[ M]_r$.
	Therefore, for any $t\in[0,T]$  we have 
	$(Y_te^{-\hat V_t})^2\leq Y_t^2e^{-\hat V_t}\leq Z+2\left|\int_0^te^{-\hat V_r}Y_{r^-}dM_r\right|$.
	Thanks to Propositions \ref{L1} and \ref{L2}, $Z$ is integrable, so we can conclude by Lemma 3.18 in \cite{paper1preprint} applied to the process $Ye^{-\hat V}$, and the fact that $\hat V$ is bounded.

	Since $Y$ is c\`adl\`ag progressively measurable, $\underset{t\in[0,T]}{\text{sup }}|Y_t|\in L^2$ and since $\hat V$ is bounded, it is clear that $Y\in \mathcal{L}^{2,cadlag}(d\hat V\otimes d\mathbbm{P})$ and the corresponding class $\dot{Y}$ belongs to  
$L^{2,cadlag}(d\hat V\otimes d\mathbbm{P})$. 
\end{proof}

Thanks to Propositions \ref{L2} and \ref{supY}, we are allowed to 
introduce the following.
\begin{notation}\label{contraction}
We denote by $\Phi$ the operator which associates to a couple 
$(\dot{U},N)$ the couple $(\dot{Y}, M)$.
\begin{equation*}
\Phi: \begin{array}{rcl}
L^2(d\hat V\otimes d\mathbbm{P})\times \mathcal{H}^2_0 &\longrightarrow& L^{2,cadlag}(d\hat V\otimes d\mathbbm{P})\times \mathcal{H}^2_0\\
(\dot{U},N) &\longmapsto& (\dot{Y},M).
\end{array}
\end{equation*}
\end{notation}

\begin{proposition}\label{FixedPoint}
The mapping $(Y,M)\longmapsto(\dot{Y},M)$ induces a bijection between the set of solutions of $BSDE(\xi,\hat{f},\hat V,\hat{M})$ and the set of fixed points of $\Phi$.
\end{proposition}
\begin{proof}First, let $(U,N)$ be a solution of $BSDE(\xi,\hat{f},V,\hat{M})$, let $(\dot{Y},M):=\Phi(\dot{U},N)$ and let $Y$ be the reference c\`adl\`ag process associated to $U$ as in Definition \ref{defYM}. By this same definition,  $M$ is the c\`adl\`ag version of 
	\\
	$t\mapsto \mathbbm{E}\left[\xi+\int_0^T\hat{f}\left(r,\cdot,U_r,\frac{d\langle N,\hat{M} \rangle}{d\hat V}(r)\right)d\hat V_r\middle|\mathcal{F}_t\right]-\mathbbm{E}\left[\xi+\int_0^T\hat{f}\left(r,\cdot,U_r,\frac{d\langle N,\hat{M} \rangle}{d\hat V}(r)\right)d\hat V_r\middle|\mathcal{F}_0\right]$, but by Proposition \ref{BSDEexpectations}, so is $N$, meaning $M=N$. Again by Definition \ref{defYM}, 
	$Y =\xi + \int_{\cdot}^T \hat{f}\left(r,\cdot,U_r,\frac{d\langle N,\hat{M} \rangle}{d\hat V}(r)\right)d\hat V_r -(N_T-N_{\cdot})$ which is equal to $U$ thanks to \eqref{BSDEcadlag}, consequently $Y=U$
	in the sense of indistinguishability. In particular, $\dot{U}=\dot{Y}$, implying $(\dot{U},N)=(\dot{Y},M)=\Phi(\dot{U},N)$. Therefore, 
the mapping $(Y,M)\longmapsto(\dot{Y},M)$  does indeed map the set of solutions of $BSDE(\xi,\hat{f},V,\hat{M})$ into the set of fixed points of $\Phi$.
	
	The map $\Phi$ is surjective. Indeed let $(\dot{U},N)$ be a fixed point of $\Phi$, the couple $(Y,M)$ of Definition \ref{defYM} satisfies  
	$Y =\xi + \int_{\cdot}^T \hat{f}\left(r,\cdot,U_r,\frac{d\langle N,\hat{M} \rangle}{d\hat V}(r)\right)d\hat V_r -(M_T-M_{\cdot})$
	in the sense of indistinguishability, and $(\dot{Y},M)=\Phi(\dot{U},N)=(\dot{U},N)$, so by Lemma 3.9 in \cite{paper1preprint},
 $\int_{\cdot}^T \hat{f}\left(r,\cdot,Y_r,\frac{d\langle M,\hat{M} \rangle}{d\hat V}(r)\right)d\hat V_r$ and $\int_{\cdot}^T \hat{f}\left(r,\cdot,U_r,\frac{d\langle N,\hat{M}\rangle}{d\hat V}(r)\right)d\hat V_r$ are indistinguishable  and
	$Y =\xi + \int_{\cdot}^T \hat{f}\left(r,\cdot,Y_r,\frac{d\langle M,\hat{M} \rangle}{d\hat V}(r)\right)d\hat V_r -(M_T-M_{\cdot})$, meaning that
	$(Y,M)$ (which is a preimage of $(\dot{U},N)$) solves $BSDE(\xi,\hat{f},V,\hat{M})$.
	
	We finally show that it is injective. Let us consider two solutions $(Y,M)$ and $(Y',M)$ of $BSDE(\xi,\hat{f},V,\hat{M})$ with  $\dot{Y}=\dot{Y'}$. By Lemma 3.9 in \cite{paper1preprint}
	the processes $\int_{\cdot}^T\hat{f}\left(r,\cdot,Y_r,\frac{d\langle M,\hat{M} \rangle}{d\hat V}(r)\right)d\hat V_r$ and $\int_{\cdot}^T\hat{f}\left(r,\cdot,Y'_r,\frac{d\langle M,\hat{M} \rangle}{d\hat V}(r)\right)d\hat V_r$ are indistinguishable, so taking \eqref{BSDEcadlag} into account, we have $Y=Y'$.
\end{proof}

\begin{proposition}\label{realmart}
	Let $\lambda\in\mathbbm{R}$, let $(\dot U,N)$, $(\dot U',N') 
\in	L^2(d\hat V\otimes d\mathbbm{P})\times \mathcal{H}^2_0$, let $(\dot Y,M)$, $(\dot Y',M')$ be their images through $\Phi$ and let $Y,Y'$ be the c\`adl\`ag representatives of $\dot Y$, $\dot Y'$ introduced in Definition \ref{defYM}. Then $\int_0^{\cdot}e^{\lambda \hat V_r}Y_{r^-}dM_r$, $\int_0^{\cdot}e^{\lambda \hat V_r}Y'_{r^-}dM'_r$, $\int_0^{\cdot}e^{\lambda \hat V_r}Y_{r^-}dM'_r$ and $\int_0^{\cdot}e^{\lambda \hat V_r}Y'_{r^-}dM_r$ are martingales.
\end{proposition}

\begin{proof}
	$\hat V$ is bounded and  thanks to Proposition \ref{supY} we know that $\underset{t\in[0,T]}{\text{sup }} |Y_t|$ and $\underset{t\in[0,T]}{\text{sup }} |Y'_t|$ are $L^2$. Moreover, since $M$ and $M'$ are square integrable, the statement
	yields therefore as a consequence of Lemma 3.17 in \cite{paper1preprint}.
\end{proof}

Starting from now, if $(\dot{Y},M)$ is the image by $\Phi$ of some 
\\
$(\dot{U},N)\in L^2(d\hat V\otimes d\mathbbm{P})\times \mathcal{H}^2_0$,
 by default, we will always refer to  the c\`adl\`ag reference process 
$Y$ of $\dot{Y}$ defined in Definition \ref{defYM}.

For any $\lambda\geq 0$, 
 on $L^2(d\hat V\otimes d\mathbbm{P})\times\mathcal{H}^2_0$
we define the  norm
\\
$\|(\dot Y,M)\|_{\lambda}^2 :=\mathbbm{E}\left[\int_0^T e^{\lambda \hat V_r}Y_r^2d\hat V_r\right] + \mathbbm{E}\left[\int_0^T e^{\lambda \hat V_r}d\langle M\rangle_r\right]$.
Since $\hat V$ is bounded, these norms are all equivalent.

\begin{proposition}\label{ProofContraction}
There exists  $\lambda>0$ such that for any 
\\
$(\dot U,N)\in L^2(d\hat V\otimes d\mathbbm{P})\times\mathcal{H}^2_0$, $\left\|\Phi(\dot U,N)\right\|^2_{\lambda}\leq \frac{1}{2}\left\|(\dot U,N)\right\|^2_{\lambda}$. In particular, $\Phi$ is a contraction in 
$L^2(d\hat V\otimes d\mathbbm{P})\times\mathcal{H}^2_0$ for the norm $\|\cdot\|_{\lambda}$. 
\end{proposition}
\begin{proof}
	
	Let $(\dot U,N)$ and $(\dot U',N')$ be two couples 
belonging to $L^2(d\hat V\otimes d\mathbbm{P})\times\mathcal{H}^2_0$, let $(\dot Y,M)$ and $(\dot Y',M')$ be their images via $\Phi$ and let $Y,Y'$ be the c\`adl\`ag reference process of $\dot Y$, $\dot Y'$ introduced in Definition \ref{defYM}. We will write $\bar{Y}$ for $Y-Y'$ and we adopt a similar notation  for  other processes. We will also write  
	\\
	$ \bar{f}_t := \hat{f}\left(t,\cdot,U_t,\frac{d\langle N,\hat{M} \rangle}{d\hat V}(t)\right) -\hat{f}\left(t,\cdot,U'_t,\frac{d\langle N',\hat{M} \rangle}{d\hat V}(t)\right).$
	
	By additivity,
	we have $d\bar Y_t=-\bar{f}_td\hat V_t +d\bar M_t$.
	Since $\bar{Y}_T = \xi - \xi = 0$, applying the integration by parts formula to $\bar{Y}_t^2e^{\lambda \hat V_t}$ between $0$ and $T$ we get
	\begin{equation*}
	\bar{Y}_0^2 - 2\int_0^T e^{\lambda \hat V_r}\bar{Y}_r\bar{f}_rd\hat V_r + 2\int_0^T e^{\lambda \hat V_r}\bar{Y}_{r^-}d\bar{M}_r +\int_0^T e^{\lambda \hat V_r}d[\bar{M}]_r + \lambda\int_0^T e^{\lambda \hat V_r}\bar{Y}^2_rd\hat V_r=0.
	\end{equation*}
	
	Since, by Proposition \ref{realmart}, the stochastic integral with respect
	to $\bar M$ is a real martingale,
	by taking the expectations we get
	
	\begin{equation*}
	\mathbbm{E}\left[\bar{Y}_0^2\right] - 2 \mathbbm{E}\left[\int_0^T e^{\lambda \hat V_r}\bar{Y}_r\bar{f}_rd\hat V_r\right] + \mathbbm{E}\left[\int_0^T e^{\lambda \hat V_r}d\langle \bar{M}\rangle_r\right] + \lambda\mathbbm{E}\left[\int_0^T e^{\lambda \hat V_r}\bar{Y}^2_rd\hat V_r\right]=0.
	\end{equation*}
	So by re-arranging previous expression, by the Lipschitz condition on $\hat{f}$ stated in Hypothesis
	\ref{HypBSDE}, by the linearity of the Radon-Nikodym derivative and by Lemma \ref{ZinL2},  we get 
	
	\begin{equation*}
	\begin{array}{lll}
	& &\lambda \mathbbm{E}\left[\int_0^T e^{\lambda \hat V_r}\bar{Y}^2_rd\hat V_r\right] + \mathbbm{E}\left[\int_0^T e^{\lambda \hat V_r}d\langle \bar{M}\rangle_r\right]\\
	&\leq&2 \mathbbm{E}\left[\int_0^T e^{\lambda \hat V_r}|\bar{Y}_r||\bar{f}_r|d\hat V_r\right]\\
	& \leq & 2K^Y\mathbbm{E}\left[\int_0^T e^{\lambda \hat V_r}|\bar{Y}_r||\bar{U}_r|d\hat V_r\right]+2K^Z\underset{i=1}{\overset{d}{\sum}}\mathbbm{E}\left[\int_0^T e^{\lambda \hat V_r}|\bar{Y}_r|\left|\frac{d\langle \bar{N},\hat{M}^i\rangle }{d\hat V}(r)\right|d\hat V_r\right] \\
	&\leq& (K^Y\alpha +d K^Z\beta)\mathbbm{E}\left[\int_0^T e^{\lambda \hat V_r}\bar{Y}_r^2d\hat V_r\right] + \frac{K^Y}{\alpha}\mathbbm{E}\left[\int_0^T e^{\lambda \hat V_r}\bar{U}_r^2d\hat V_r\right] \\
	&&+ \frac{K^Z}{\beta}\underset{i=1}{\overset{d}{\sum}}\mathbbm{E}\left[\int_0^T e^{\lambda \hat V_r}\left(\frac{d\langle \bar{N},\hat{M}^i\rangle }{d\hat V}(r)\right)^2d\hat V_r\right]\\
	&\leq&(K^Y\alpha +d K^Z\beta)\mathbbm{E}\left[\int_0^T e^{\lambda \hat V_r}\bar{Y}_r^2d\hat V_r\right] + \frac{K^Y}{\alpha}\mathbbm{E}\left[\int_0^T e^{\lambda \hat V_r}\bar{U}_r^2d\hat V_r\right] \\
	&&+ \frac{CdK^Z}{\beta}\mathbbm{E}\left[\int_0^T e^{\lambda \hat V_r}d\langle \bar{N}\rangle_r\right],\\
	\end{array}
	\end{equation*}
	for some positive $C$ and any positive $\alpha$ and $\beta$.
The latter equality holds by Hypothesis \ref{HypBSDE} 4.
 Then we pick $\alpha = 2K^Y$ and $\beta = 2CdK^Z$, which gives us
	
	\begin{equation*} 
	\begin{array}{rcl}
	&&\lambda \mathbbm{E}\left[\int_0^T e^{\lambda \hat V_r}\bar{Y}^2_rd\hat V_r\right] + \mathbbm{E}\left[\int_0^T e^{\lambda \hat V_r}d\langle \bar{M}\rangle_r\right] \\
	&\leq& 2((K^Y)^2 + C(dK^Z)^2)\mathbbm{E}\left[\int_0^T e^{\lambda \hat V_r}\bar{Y}_r^2d\hat V_r\right] \\
	&+& \frac{1}{2}\mathbbm{E}\left[\int_0^T e^{\lambda \hat V_r}\bar{U}_r^2d\hat V_r\right] 
	+ \frac{1}{2}\mathbbm{E}\left[\int_0^T e^{\lambda \hat V_r}d\langle \bar{N}\rangle_r\right].
	\end{array}
	\end{equation*}
	We choose now $\lambda = 1 + 2((K^Y)^2 + C(dK^Z)^2)$ and we get
	\begin{equation} \label{E152}
	\begin{array}{rcl}
	&&\mathbbm{E}\left[\int_0^T e^{\lambda \hat V_r}\bar{Y}^2_rd\hat V_r\right] + \mathbbm{E}\left[\int_0^T e^{\lambda \hat V_r}d\langle\bar{M}\rangle_r\right]\\
	&\leq& \frac{1}{2}\mathbbm{E}\left[\int_0^T e^{\lambda \hat V_r}\bar{U}_r^2d\hat V_r\right]
	+ \frac{1}{2}\mathbbm{E}\left[\int_0^T e^{\lambda \hat V_r}d\langle \bar{N}\rangle_r\right],
	\end{array}
	\end{equation}
	which proves the contraction for the norm  $\|\cdot\|_{\lambda}$.
\end{proof}

\begin{prooff}\\ of Theorem \ref{uniquenessBSDE}.
	
	The space $L^2(d\hat V\otimes d\mathbbm{P})\times \mathcal{H}^2_0$ is complete and $\Phi$ defines on it a contraction for the norm $\|(\cdot,\cdot)\|_{\lambda}$ for some $\lambda>0$, so $\Phi$ has a unique fixed point in 
	\\
	$L^2(d\hat V\otimes d\mathbbm{P})\times \mathcal{H}^2_0$. Then by Proposition \ref{FixedPoint}, $BSDE(\xi,\hat{f},V,\hat{M})$ has a unique solution. 
\end{prooff}

\begin{remark}\label{RealMart}
	Let $(Y,M)$ be the solution of $BSDE(\xi,\hat{f},V,\hat{M})$ and $\dot{Y}$ the class of $Y$ in $L^2(d\hat V\otimes d\mathbbm{P})$. Thanks to Proposition \ref{FixedPoint}, we know  that 
	\\
	$(\dot{Y},M)=\Phi(\dot{Y},M)$ and therefore by Propositions \ref{supY} and \ref{realmart} that $\underset{t\in[0,T]}{\text{ sup }}|Y_t|$ is $L^2$ and that $\int_0^{\cdot}Y_{r^-}dM_r$ is a real martingale.
\end{remark}

The  lemma below  shows that, in order to check if a couple
$(Y,M)$ is the solution of $BSDE(\xi,\hat{f},V,\hat{M})$, it is not necessary 
to verify the square integrability of $Y$ since
it will be automatically fulfilled.
\begin{lemma}\label{LED+Pext}
 We consider  $(\xi,\hat{f},V,\hat{M})$ such that $\xi,\hat{M}$ satisfy items 1., 2. 
 of Hypothesis \ref{HypBSDE} but where item 3. is replaced by the
weaker following hypothesis on $\hat f$. There exists $C>0$ such that $\mathbbm{P}$ a.s., for all $t,y,z$, 
\begin{equation}
|\hat{f}(t,\omega,y,z)|\leq C(1+|y|+\|z\|).
\end{equation}
Assume that there exists a c\`adl\`ag adapted process $Y$ with $Y_0\in L^2$ , and $M\in\mathcal{H}^2_0$ such that 
	\begin{equation}\label{EqLEDPext}
	Y = \xi + \int_{\cdot}^T
	\hat{f}\left(r,\cdot,Y_r,\frac{d\langle M,\hat{M}\rangle}{d\hat V}(r)\right) d\hat V_r - (M_T-M_{\cdot}),
	\end{equation}
	in the sense of indistinguishability. Then $\underset{t\in[0,T]}{\text{sup }}|Y_t|$ is $L^2$.
        In particular, 
	$Y\in \mathcal{L}^2(d\hat V\otimes d\mathbbm{P})$ and if $(\xi,\hat{f},V,\hat{M})$ satisfies Hypothesis \ref{HypBSDE}, 
then $(Y,M)$ is the unique solution of $BSDE(\xi,\hat{f},V,\hat{M})$ in the sense of Definition \ref{firstdefBSDE}.
	
	On the other hand if $(Y,M)$ satisfies \eqref{EqLEDPext} 
	on $[s,T]$ with $s<T$,  $Y_s\in L^2$ and  $M_s=0$ then $\underset{t\in[s,T]}{\text{sup }}|Y_t|$ is $L^2$. 
Consequently if $(\xi,\hat{f},V,\hat{M})$ satisfies Hypothesis \ref{HypBSDE} and if
 we denote $(U,N)$ the unique solution of $BSDE(\xi,\hat{f},V,\hat{M})$, then $(Y,M)$ and $(U,N_{\cdot}-N_s)$ are indistinguishable on $[s,T]$.
\end{lemma}
\begin{proof}

	Let $\lambda>0$ and $t\in[0,T]$. By integration by parts formula applied to $Y^2e^{-\lambda \hat V}$ between $0$ and $t$ we get
	\begin{equation*}
	\begin{array}{rcl}
	Y^2_te^{-\lambda \hat V_t}-Y_0^2 &=& -2\int_0^te^{-\lambda \hat V_r}Y_r\hat{f}\left(r,\cdot,Y_r,\frac{d\langle M,\hat{M}\rangle}{d\hat V}(r)\right)d\hat V_r +2\int_0^te^{-\lambda \hat V_r}Y_{r^-}dM_r \\
	& & +\int_0^te^{-\lambda \hat V_r}d[M]_r-\lambda\int_0^te^{-\lambda \hat V_r}Y_r^2
d \hat V_r.

	\end{array}
	\end{equation*}
		By re-arranging the terms  and using the Lipschitz conditions
stated in item 3. of in Hypothesis \ref{HypBSDE}, we get
	\begin{equation*}
	\begin{array}{rcl}
	&& Y^2_te^{-\lambda \hat V_t}+\lambda\int_0^te^{-\lambda \hat V_r}Y_r^2d\hat V_r\\
	&\leq& Y_0^2 + 2\int_0^te^{-\lambda \hat V_r}|Y_r||\hat{f}|\left(r,\cdot,Y_r,\frac{d\langle M,\hat{M}\rangle}{d\hat V}(r)\right)d\hat V_r+2\left|\int_0^te^{-\lambda \hat V_r}Y_{r^-}dM_r\right|\\
	&& +\int_0^te^{-\lambda \hat V_r}d[M]_r\\
	& \leq& Y_0^2 + \int_0^te^{-\lambda \hat V_r}\hat{f}^2(r,\cdot,0,0)d\hat V_r+(2K^Y+1+K^Z)\int_0^te^{-\lambda \hat V_r}Y_r^2d\hat V_r\\
	& &+K^Z\underset{i=1}{\overset{d}{\sum}}\int_0^te^{-\lambda \hat V_r}\left(\frac{d\langle M,\hat{M}^i\rangle}{d\hat V}(r)\right)^2d\hat V_r+2\left|\int_0^te^{-\lambda \hat V_r}Y_{r^-}dM_r\right| +\int_0^te^{-\lambda \hat V_r}d[M]_r. 
	\end{array}
	\end{equation*}
	Picking $\lambda = 2K^Y+1+K^Z$ and using Lemma \ref{ZinL2}, this gives 
	\begin{equation*}
	\begin{array}{rcl}
	Y^2_te^{-\lambda \hat V_t} &\leq& Y_0^2 + \int_0^te^{-\lambda \hat V_r}|\hat{f}|^2(r,\cdot,0,0)d\hat V_r+K^ZC\int_0^te^{-\lambda \hat V_r}d\langle M\rangle_r\\
	&&+2\left|\int_0^te^{-\lambda \hat V_r}Y_{r^-}dM_r\right| +\int_0^te^{-\lambda \hat V_r}d[M]_r,
	\end{array}
	\end{equation*}
	for some $C>0$. Since $\hat V$ is bounded, there is a constant $C'>0$, such that for any $t\in[0,T]$
	\begin{equation*}
	Y^2_t \leq C'\left(Y_0^2 + \int_0^T|\hat{f}|^2(r,\cdot,0,0)d\hat V_r+\langle M\rangle_T +[M]_T+\left|\int_0^tY_{r^-}dM_r\right|\right).
	\end{equation*}
	By Hypothesis \ref{HypBSDE}, $Y_0\in L^2$ and $M\in\mathcal{H}^2$, the first four terms on the right-hand side are integrable so that we can conclude by Lemma 3.18 in \cite{paper1preprint}.
	
	An analogous proof also holds on the interval $[s,T]$ 
	taking into account  Remark \ref{BSDESmallInt}.
In particular, if  $(U,N)$ is the unique solution  of  
  $BSDE(\xi,\hat{f},V,\hat{M})$ then
$(U,N - N_s)$ is a solution on $[s,T]$.
  The final statement result follows by the uniqueness argument of Remark \ref{BSDESmallInt}.
\end{proof}

\begin{notation}\label{DefPicard}
	Let $\Phi: L^2(d\hat V\otimes d\mathbbm{P})\times\mathcal{H}^2_0$ be the operator
	introduced in  Notation \ref{contraction}.
	 
	In the sequel we will not distinguish between a couple
	$(\dot Y,M)$ in $L^2(d\hat V\otimes d\mathbbm{P})\times\mathcal{H}^2_0$ 
	and $(Y,M)$, where $Y$ is the reference c\`adl\`ag process of $\dot Y$,
	according to Definition \ref{defYM}. We then convene the following.
	\begin{enumerate}
		\item $(Y^{0},M^{0}):=(0,0)$;
		\item $\forall k\in\mathbbm{N}^*:(Y^{k},M^{k}):=\Phi(Y^{k-1},M^{k-1})$,
	\end{enumerate}
	meaning that for $k\in\mathbbm{N}^*$, $(Y^{k},M^{k})$ is the solution of the BSDE
	\begin{equation}\label{defYk}
	Y^{k} = \xi + \int_{\cdot}^T \hat{f}\left(r,\cdot,Y^{k-1},\frac{d\langle M^{k-1},\hat{M}\rangle}{d\hat V}(r)\right)d\hat V_r  -(M^{k}_T - M^{k}_{\cdot}).
	\end{equation}
\end{notation}

\begin{definition}
	The processes $(Y^k,M^k)_{k\in\mathbbm{N}}$ will be called the Picard
 iterations associated to $BSDE(\xi,\hat{f},\hat V,\hat{M})$.
\end{definition}

We know that $\Phi$ is a contraction in $L^2(d\hat V\otimes d\mathbbm{P}^{s,x})\times\mathcal{H}^2_0$ for a certain norm, so that $(Y^k,M^k)$ tends to $(Y,M)$ in this topology.
The proposition below also shows an a.e. corresponding convergence,
adapting the techniques of
Corollary 2.1 in \cite{el1997backward}.
\begin{proposition}\label{cvdt}
	$Y^{k}\underset{k\rightarrow \infty}{\longrightarrow} Y\quad d\hat V\otimes d\mathbbm{P}$ a.e.  and for any $i\in[\![ 1;d]\!]$, 
	\\
	$\frac{d\langle M^{k},\hat{M}^i\rangle}{d\hat V}\underset{k\rightarrow \infty}{\longrightarrow}\frac{d\langle M,\hat{M}^i\rangle}{d\hat V}\quad d\hat V\otimes d\mathbbm{P}$ a.e.
	
\end{proposition}

\begin{proof}
	
	For any $i\in[\![ 1;d]\!]$ and $k\in\mathbbm{N}$ we set $Z^{i,k}:=\frac{d\langle M^{k},\hat{M}^i\rangle}{d\hat V}$ and $Z^{i}:=\frac{d\langle M,\hat{M}^i\rangle}{d\hat V}$.
	By Proposition \ref{ProofContraction},  
	there exists $\lambda>0$ such that for any $k\in\mathbbm{N}^*$
	\begin{equation*}
	\begin{array}{rl}
	&\mathbbm{E}\left[\int_0^Te^{-\lambda \hat V_r}|Y^{k+1}_r-Y^{k}_r|^2d\hat V_r + \int_0^Te^{-\lambda \hat V_r}d\langle M^{k+1}-M^{k}\rangle_r\right]\\
	\leq &\frac{1}{2}\mathbbm{E}\left[\int_0^Te^{-\lambda \hat V_r}|Y^{k}_r-Y^{k-1}_r|^2d\hat V_r + \int_0^Te^{-\lambda \hat V_r}d\langle M^{k}-M^{k-1}\rangle_r\right],
	\end{array}
	\end{equation*}
        consequently
	\begin{equation} \label{E514}
  \begin{array}{rl}
	&\underset{k\geq 0}{\sum}\mathbbm{E}\left[\int_0^Te^{-\lambda \hat V_r}|Y^{k+1}_r-Y^{k}_r|^2d\hat V_r\right] + \mathbbm{E}\left[\int_0^Te^{-\lambda \hat V_r}d\langle M^{k+1}-M^{k}\rangle_r\right]\\
	\leq& \underset{k\geq 0}{\sum}\frac{1}{2^k}\left(\mathbbm{E}\left[\int_0^Te^{-\lambda \hat V_r}|Y^{1 }_r|^2d\hat V_r\right] + \mathbbm{E}\left[\int_0^Te^{-\lambda \hat V_r}d\langle M^{1}\rangle_r\right]\right)
	< \infty.
	\end{array}
	\end{equation}
	For every fixed $(i,k)$, we have
	\\
	$Z^{i,k+1}_r-Z^{i,k}_r=\frac{d\langle M^{k+1}-M^{k},\hat{M}^i\rangle}{d\hat V}$. Therefore combining equation \eqref{E514} and Lemma \ref{ZinL2}, we get 
	\\
	$\underset{k\geq 0}{\sum}\left(\mathbbm{E}\left[\int_0^Te^{-\lambda \hat V_r}|Y^{k+1}_r-Y^{k}_r|^2d\hat V_r\right] + \underset{i=1}{\overset{d}{\sum}}\mathbbm{E}\left[\int_0^Te^{-\lambda \hat V_r}|Z^{i,k+1}_r-Z^{i,k}_r|^2d\hat V_r\right]\right)<\infty$.
	So by Fubini's theorem we have 
	\\
	$\mathbbm{E}\left[\int_0^Te^{-\lambda \hat V_r}\left(\underset{k\geq 0}{\sum}\left(|Y^{k+1}_r-Y^{k}_r|^2+\underset{i=1}{\overset{d}{\sum}}|Z^{i,k+1}_r-Z^{i,k}_r|^2\right)\right)d\hat V_r\right]<\infty$.
	\\
	\\
	Consequently  the sum 
	$\underset{k\geq 0}{\sum}\left(|Y^{k+1}_r(\omega)-Y^{k}_r(\omega)|^2 + \underset{i=1}{\overset{d}{\sum}}|Z_r^{i,k+1}(\omega)-Z^{i,k}_r(\omega)|^2\right)$
	is finite on a set of full $d\hat V\otimes d\mathbbm{P}$-measure. So on this set, the sequence $(Y^{k}_t(\omega),(Z^{i,k}_t(\omega))_{i\in[\![ 1;d]\!]})$ converges, and the limit is necessarily equal to $(Y_t(\omega),(Z^i_t(\omega))_{i\in[\![ 1;d]\!]}) \quad d\hat V\otimes d\mathbbm{P}$ a.e. Indeed, as we have mentioned in the lines before the statement of the
	present Proposition \ref{cvdt}, we already know that $Y^k$
converges to $Y$ in $L^2(d\hat V\otimes d\mathbbm{P})$. Since by Lemma
 \ref{ZinL2}, $\mathbbm{E}\left[\int_0^Te^{-\lambda \hat V_r}|Z^{i,k}_r-Z^{i}_r|^2d\hat V_r\right]\leq C\mathbbm{E}\left[\int_0^Te^{-\lambda \hat V_r}d\langle M^{k}-M\rangle_r\right]$, for every $(i,k)$, 
	where $C$ is a positive constant which does not depend on $(i,k)$, 
the convergence of $M^k$ to $M$ in $\mathcal{H}^2_0$ also implies the convergence of $Z^{i,k}$ to $Z^i$ in $L^2(d\hat V\otimes d\mathbbm{P})$.
	
\end{proof}

 \section{Proof of Theorem \ref{Defuv}}\label{B}

\begin{lemma} \label{L41}
Let $\tilde{f}\in\mathcal{L}^2_X$.
For every
$(s,x)\in[0,T]\times E$, 
let $(\tilde Y^{s,x},\tilde M^{s,x})$ be the unique (by Theorem \ref{uniquenessBSDE} and Remark \ref{BSDESmallInt}) solution of 
\begin{equation}\label{FBSDEftildeBis}
\tilde Y^{s,x}_{\cdot} = g(X_T) + \int_{\cdot}^T \mathds{1}_{[s,T]}(r)\tilde f\left(r,X_r\right)dV_r  -(\tilde M^{s,x}_T - \tilde M^{s,x}_{\cdot})
\end{equation}
in $\left(\Omega,\mathcal{F}^{s,x},(\mathcal{F}^{s,x}_t)_{t\in[0,T]},\mathbbm{P}^{s,x}\right)$.
  Then there exist $\tilde u\in\mathcal{D}(\mathfrak{a})$

 such that for any $(s,x)\in[0,T]\times E$

\begin{equation*}
\left\{\begin{array}{rcl}
     \forall t\in [s,T]: \tilde Y^{s,x}_t &=& \tilde u(t,X_t)\quad \mathbbm{P}^{s,x}\text{a.s.}  \\
     \tilde M^{s,x}&=&M[\tilde u]^{s,x}
\end{array}\right.
\end{equation*}
and in particular $\frac{d\langle \tilde M^{s,x},M[\psi]^{s,x}\rangle}{dV}=\mathfrak{G}^{\psi}(\tilde u)(\cdot,X_{\cdot})$  $dV\otimes d\mathbbm{P}^{s,x}$  a.e. on $[s,T]$.

\end{lemma}
\begin{proof}
We set  $\tilde u:(s,x)\mapsto \mathbbm{E}^{s,x}\left[g(X_T) + \int_s^T \tilde{f}\left(r,X_r\right)dV_r\right]$ 
which is Borel by Proposition A.10  and Lemma A.11 in \cite{paper2}. Therefore 
by the Markov property  (see e.g. (3.4) in \cite{paperAF}),
 for every fixed $t \in[s,T]$ we have 
$\mathbb {P}^{s,x}$- a.s. 
\begin{equation*}
\begin{array}{rcccl}
    \tilde u(t,X_t) &=& \mathbbm{E}^{t,X_t}\left[g(X_T) + \int_t^T \tilde{f}\left(r,X_r\right)dV_r\right]
     &=& \mathbbm{E}^{s,x}\left[g(X_T) + \int_t^T \tilde{f}\left(r,X_r\right)dV_r\middle|\mathcal{F}_t\right]\\
     &=& \mathbbm{E}^{s,x}\left[\tilde Y^{s,x}_t+(\tilde M^{s,x}_T-\tilde M^{s,x}_t)|\mathcal{F}_t\right]
     &=& \tilde Y^{s,x}_t.
\end{array}
\end{equation*}
By \eqref{FBSDEftildeBis} we have $d\tilde Y^{s,x}_t=-\tilde{f}(t,X_t)dV_t+d\tilde M^{s,x}_t$, so for every  fixed $t \in[s,T]$,   $\tilde u(t,X_t)= \tilde u(s,x)-\int_s^t\tilde{f}(r,X_r)dV_r -\tilde M^{s,x}_t$ $\mathbb {P}^{s,x}$- a.s.  Since $\tilde M^{s,x}$ is square integrable and since  previous relation holds for any $(s,x)$ and $t$, Definition \ref{extended} implies that $\tilde u\in\mathcal{D}(\mathfrak{a})$, $\mathfrak{a}(\tilde u)=-\tilde{f}$ and $\tilde M^{s,x}=M[\tilde u]^{s,x}$ for every $(s,x)$, hence the announced results.
\end{proof}

\begin{notation}
For a fixed $(s,x)\in [0,T]\times E$, we will denote by  $(Y^{k,s,x},M^{k,s,x})_{k\in\mathbbm{N}}$ the Picard iterations associated to  $BSDE^{s,x}(f,g)$.
\end{notation}

\begin{proposition} \label{P511}
	For each $k\in\mathbbm{N}$, there exists  $u_k\in\mathcal{D}(\mathfrak{a})$, such that for every $(s,x)\in[0,T]\times E$
	\begin{equation}\label{defuk}
	\left\{\begin{array}{rcl}
	\forall t\in [s,T]: Y^{k,s,x}_t &=& u_k(t,X_t) \quad \mathbbm{P}^{s,x}a.s.  \\
	M^{k,s,x}&=&M[u_k]^{s,x}
	\end{array}\right.
	\end{equation}
\end{proposition}
\begin{remark} \label{R511}
In particular, \eqref{defuk} implies that
 $\frac{d\langle M^{k,s,x},M[\psi]^{s,x}\rangle}{dV}=\mathfrak{G}^{\psi}(u_k)(\cdot,X_{\cdot})$  $dV\otimes d\mathbbm{P}^{s,x}$  a.e. on $[s,T]$.
\end{remark}

\begin{proof}
	We proceed by induction on $k$.
	It is clear that $u_0=0$ satisfies the assertion for $k=0$. 

	Now let us assume that the function $u_{k-1}$ exists, for some integer 
	$k \ge 1$, 
	satisfying \eqref{defuk} 
and in particular Remark \ref{R511},
for $k$ replaced with $k-1$.
	
	We fix $(s,x)\in[0,T]\times E$. By Lemma \ref{ModifImpliesdV},
	$(Y^{k-1,s,x},\frac{d\langle M^{k-1,s,x},M[\psi]^{s,x}\rangle}{dV})=(u_{k-1},\mathfrak{G}^{\psi}(u_{k-1}))(\cdot,X_{\cdot})$ $dV\otimes \mathbbm{P}^{s,x}$ a.e. on $[s,T]$. Therefore by \eqref{defYk}, on $[s,T]$
	\\
	$Y^{k,s,x} = g(X_T) + \int_{\cdot}^T f\left(r,X_r,u_{k-1}(r,X_r),\mathfrak{G}^{\psi}(u_{k-1})(r,X_r)\right)dV_r  -(M^{k,s,x}_T - M^{k,s,x}_{\cdot})$.
	
	Since $\Phi^{s,x}$ maps $L^2(dV\otimes d\mathbbm{P}^{s,x})\times \mathcal{H}^2_0$ into itself (see Definition \ref{contraction}), obviously all 
	the Picard iterations 
	belong to  
	$L^2(dV\otimes d\mathbbm{P}^{s,x})\times \mathcal{H}^2_0$. 
	In particular, by Lemma \ref{ZinL2}
	$Y^{k-1,s,x}$ and for every $i\in[\![ 1;d]\!]$,  $\frac{d\langle M^{k-1,s,x},M[\psi_i]^{s,x}\rangle}{dV}$  belong to
	\\
	$\mathcal{L}^2(dV\otimes d\mathbbm{P}^{s,x})$. So, 
	by recurrence assumption 
	on  $ u_{k-1}$, it follows that
	$u_{k-1}$ and for any $i\in[\![ 1;d]\!]$, $\mathfrak{G}^{\psi_i}(u_{k-1})$ belong to $\mathcal{L}^2_X$.

	Combining $H^{mom}(\zeta,\eta)$ and the growth condition of $f$
(item 3.)
 in $H^{lip}(\zeta,\eta)$ (see Hypotheses \ref{HypMom} and \ref{Hpq}),
one shows that $f(\cdot,\cdot,0,0)$ also belongs to $\mathcal{L}^2_X$.
	Therefore thanks to the Lipschitz conditions on $f$ assumed in $H^{lip}(\zeta,\eta)$, we have $f(\cdot,\cdot,u_{k-1},\mathfrak{G}^{\psi}(u_{k-1})) \in \mathcal{L}^2_X$.
	
	The existence of $u_k$  now comes from Lemma \ref{L41} applied to $\tilde f:=f(\cdot,\cdot,u_{k-1},\mathfrak{G}^{\psi}(u_{k-1}))$, which establishes the induction step
	for a general $k$ and allows to conclude the proof.
	
\end{proof}

\begin{prooff}\\ of Theorem \ref{Defuv}.
We  set $\bar{u}:=\underset{k\in\mathbbm{N}}{\text{limsup }}u_k$,
 in the sense that for any $(s,x)\in[0,T]\times E$, 
${\bar u}(s,x)= \underset{k\in\mathbbm{N}}{\text{limsup }}u_k(s,x)$ and $v:=\underset{k\in\mathbbm{N}}{\text{limsup }}v_k$. $\bar{u}$ and $v$ are Borel functions. Let us fix now  $(s,x)\in[0,T]\times E$.
 We know by Propositions \ref{P511}, \ref{cvdt} and Lemma \ref{ModifImpliesdV} that 
\begin{equation*}
\left\{\begin{array}{rcl}
     u_k(\cdot,X_{\cdot})&\underset{k\rightarrow\infty}{\longrightarrow}& Y^{s,x}  \quad dV\otimes d\mathbbm{P}^{s,x}\text{ a.e. on }[s,T]\\
     \mathfrak{G}^{\psi}(u_{k})(\cdot,X_{\cdot})&\underset{k\rightarrow\infty}{\longrightarrow}& Z^{s,x}  \quad dV\otimes d\mathbbm{P}^{s,x}\text{ a.e. on }[s,T],
\end{array}\right.
\end{equation*}
where $Z^{s,x} := \frac{d\langle M^{s,x},M[\psi]^{s,x}\rangle}{dV}$.
Therefore, and on the subset of $[s,T]\times E$ of full $dV\otimes d\mathbbm{P}^{s,x}$-measure on which these convergences hold, we have 
\begin{equation}\label{E420}
\left\{\begin{array}{rcccccl}
     \bar{u}(t,X_t(\omega))&=&\underset{k\in\mathbbm{N}}{\text{limsup }}u_k(t,X_t(\omega))&=&\underset{k\in\mathbbm{N}}{\text{lim }}u_k(t,X_t(\omega)) &=& Y^{s,x}_t(\omega)  \\
     v(t,X_t(\omega))&=&\underset{k\in\mathbbm{N}}{\text{limsup }}\mathfrak{G}^{\psi}(u_{k})(t,X_t(\omega))&=&\underset{k\in\mathbbm{N}}{\text{lim  }}\mathfrak{G}^{\psi}(u_{k})(t,X_t(\omega)) &=& Z^{s,x}_t(\omega).
\end{array}\right.
\end{equation}

 Thanks to the $dV\otimes d\mathbbm{P}^{s,x}$ equalities concerning $v$ and $\bar{u}$ stated in 
\eqref{E420}, under $\mathbbm{P}^{s,x}$ we actually have
\begin{equation} \label{E421}
Y^{s,x} = g(X_T) + \int_{\cdot}^T f\left(r,X_r,\bar{u}(r,X_r),v(r,X_r)\right)dV_r  -(M^{s,x}_T - M^{s,x}_{\cdot}).
\end{equation}
Now \eqref{E421} can be considered as a BSDE where the driver does
not depend on $y$ and $z$.
Since $Y^{s,x}$ and $Z^{s,x}$ belong to $\mathcal{L}^2(dV\otimes d\mathbbm{P}^{s,x})$ 
(see Lemma \ref{ZinL2}), then by \eqref{E420}, so do $\bar{u}(\cdot,X_{\cdot})\mathds{1}_{[s,T]}$ and $v(\cdot,X_{\cdot})\mathds{1}_{[s,T]}$, 
meaning that $\bar{u}$ and $v$ belong to $\mathcal{L}^2_X$. Combining $H^{mom}(\zeta,\eta)$ and the Lipschitz condition on $f$ assumed in $H^{lip}(\zeta,\eta)$,  $f(\cdot,\cdot,\bar{u},v)$ is also proved to
belong to $\mathcal{L}^2_X$. We can therefore apply Lemma \ref{L41} to 
$\tilde{f}:=f(\cdot,\cdot,\bar{u},v)$, and conclude the proof
of the first part of the theorem.

Concerning the last statement of the
Theorem \ref{Defuv}, for any $(s,x)\in[0,T]\times E$, we have  $Y^{s,x}_s=u(s,X_s)=u(s,x)$ $\mathbbm{P}^{s,x}$ a.s. so $Y^{s,x}_s$ is $\mathbbm{P}^{s,x}$ a.s. equal to a constant and $u$ is the mapping $(s,x)\longmapsto Y^{s,x}_s$.
\end{prooff}

\end{appendices}
{\bf ACKNOWLEDGMENTS.}
The authors thank the referees for their stimulating
comments which has permitted us to increase the quality
of the paper.
The authors are also grateful to Andrea Cosso
for stimulating discussions. The research of the first named author
was provided by a PhD fellowship (AMX) of the Ecole Polytechnique. 
The work of the second named author was
partially supported by a public grant as part of the Investissement d’avenir
project, reference ANR-11-LABX-0056-LMH, LabEx LMH, in a joint call with
Gaspard Monge Program for optimization, operations research and their
interactions with data sciences.

\bibliographystyle{plain}

\bibliography{../../../../biblioPhDBarrasso_bib/biblioPhDBarrasso}

\end{document}